\documentclass[11pt]{amsart}
\usepackage[utf8]{inputenc}
\usepackage{Equivariant_Packages}

\title{On Gibbs measures and topological solitons of exterior equivariant wave maps}
\author{Bjoern Bringmann}
\date{\today}

\address
{Bjoern Bringmann\\
School of Mathematics\\
Institute for Advanced Study, Princeton, NJ 08540 \& Department of Mathematics, Princeton University, Princeton, NJ 08544}
\email{bjoern@ias.edu}

\begin{document}

\maketitle

\begin{abstract}
We consider $k$-equivariant wave maps from the exterior spatial domain $\mathbb{R}^3\backslash B(0,1)$ into the target $\mathbb{S}^3$. This model has infinitely many topological solitons $Q_{n,k}$, which are indexed by their topological degree $n\in \mathbb{Z}$. For each $n\in \mathbb{Z}$ and  $k\geq 1$, we prove the existence and invariance of a Gibbs measure supported on the homotopy class of $Q_{n,k}$. As a corollary, we obtain that soliton resolution fails for random initial data. Since soliton resolution is known for initial data in the energy space, this reveals a sharp contrast between deterministic and probabilistic perspectives. 
\end{abstract}

\tableofcontents

\section{Introduction}

The wave maps equation is one of the most prominent evolution equations of mathematical physics. We initially consider wave maps $u\colon \R^{1+3}\rightarrow \mathbb{S}^3$, which are critical points of the Lagrangian
\begin{equation}\label{intro:eq-Lagrangian}
\mathcal{L}(u,\partial_t u) := \int_{\mathbb{R}^{1+3}} \dt \dx \, 
\Big( - \big| \partial_t u \big|_g^2 + \sum_{j=1}^3 \big| \partial_{x^j} u \big|^3_g \Big). 
\end{equation}
Here, $g$ denotes the induced Riemannian metric on $\mathbb{S}^3 \subseteq \mathbb{R}^4$. In this article, we are primarily interested in a simplified model for the wave maps $u\colon \R^{1+3}\rightarrow \mathbb{S}^3$, which involves the following two simplifications: 
\begin{enumerate}[label=(\roman*)]
    \item We require that the wave map $u\colon \R^{1+3} \rightarrow \mathbb{S}^3$ is $k$-equivariant, where $k\in \mathbb{N}$. To be precise, we require that 
    \begin{equation}
            u(t,r,\omega) = \Big( \sin\big(\phi(t,r)\big) \Omega_k(\omega), \cos \big( \phi(t,r)\big) \Big), 
    \end{equation}
    where $(r,\omega)\in (0,\infty)\times \mathbb{S}^2$ are polar coordinates on $\mathbb{R}^3$, $\phi \colon \R \times (0,\infty) \rightarrow \R$ is a scalar field, and $\Omega_k \colon \mathbb{S}^2 \rightarrow \mathbb{S}^2$ is a harmonic map with eigenvalue $k (k+1)$. 
    The scalar field $\phi$ describes the angle between the wave map $u$ and the north pole $N=(0,0,0,1)$. 
    \item We replace the spatial domain $\mathbb{R}^3$ with the exterior spatial domain $\mathbb{R}^3 \backslash B(0,1)$ and impose zero Dirichlet boundary conditions. Since this breaks the scaling symmetry of the wave maps equation, it effectively turns the wave maps equation from energy-supercritical into energy-subcritical. 
\end{enumerate}
The resulting initial value problem for the scalar field $\phi=\phi(t,r)$, which is called the exterior $k$-equivariant wave maps equation, can be written as 
\begin{equation}\label{intro:eq-phi}
\begin{cases}
\begin{alignedat}{3}
\partial_t^2 \phi - \partial_r^2 \phi - \frac{2}{r} \partial_r \phi + \frac{k(k+1)}{2r^2} \sin\big( 2 \phi \big) &= 0 \hspace{15ex} &&(t,r) \in \mathbb{R} \times (1,\infty), \\ 
\phi(t,1) &=0 &&t \in \mathbb{R}, \\ 
\big( \phi, \partial_t \phi \big)(0,r) &= \big( \phi_0, \phi_1 \big)(r) &&r \in (1,\infty). 
\end{alignedat}
\end{cases}
\end{equation}
This evolution equation has the conserved energy
\begin{equation}\label{intro:eq-energy}
E_k \big( \phi, \phi_t \big) := \frac{1}{2} \int_1^\infty \dr \, r^2 \Big( (\partial_t \phi)^2 + (\partial_r \phi)^2 + \tfrac{k(k+1)}{r^2} \sin^2 \big( \phi \big)\Big). 
\end{equation}
For any smooth solution of \eqref{intro:eq-phi} with finite energy, there exists an integer $n\in \Z$ such that
\begin{equation}\label{intro:eq-n}
\lim_{r\rightarrow \infty} \phi(t,r) = n \pi
\end{equation}
for all $t\in \R$. Due to the symmetry $\phi \mapsto -\phi$ of \eqref{intro:eq-phi}, we can restrict to the case $n\geq 0$. Since $\phi$ represents the angle between the wave map $u$ and the north pole $N=(0,0,0,1)\in \mathbb{S}^3$,  the nonnegative integer represents the topological degree of the wave map. The energy space of \eqref{intro:eq-energy} can therefore be decomposed into the connected components
\begin{equation}\label{intro:eq-connected}
\Conn_{n,k} := \bigg\{ (\phi_0,\phi_1) \colon 
\int_1^\infty \dr \, r^2 \Big( (\partial_r \phi_0)^2 + \phi_1^2 \Big) <\infty, \, \phi_0(1)=0, \, \lim_{r\rightarrow \infty} \phi_0(r)= n \pi \bigg\}. 
\end{equation}
One of the most interesting features of \eqref{intro:eq-phi} is that each connected component $\Conn_{n,k}$ contains a unique minimizer of the energy $E_k$ given by $(\phi_0,\phi_1)=(Q_{n,k},0)$. The function $Q_{n,k}$ is a harmonic map, i.e., a solution of the stationary equation
\begin{equation}
- \partial_r^2 Q_{n,k} - \frac{2}{r} \partial_r Q_{n,k} + \frac{k(k+1)}{2r^2} \sin\big( 2 Q_{n,k} \big) =0. 
\end{equation}
We emphasize that this is a feature of \emph{exterior} equivariant wave maps, since equivariant wave maps on $\mathbb{R}^{1+3}$ do not have any stationary solutions with finite energy \cite{S88,SS98}. The exterior equivariant wave maps in \eqref{intro:eq-phi} were first introduced in \cite{BSSS92} as an alternative to the Skyrme equation \cite{S61}, which is a different simplification of the wave maps equation. It was further studied analytically and numerically in  \cite{BCM12}, which advertised \eqref{intro:eq-phi} as a model problem for soliton resolution. Soliton resolution was first proven for \eqref{intro:eq-phi} in the case $k=1$ in \cite{KLS14,LS13} and in the general case $k\geq 1$ in \cite{KLLS15} and is recorded in the following theorem. 

\begin{theorem}[\cite{KLLS15}]\label{intro:thm-deterministic}
For any $k\geq 1$, $n\geq 0$, and $(\phi_0,\phi_1)\in \Conn_{n,k}$, there exists a unique global solution $\phi$ of \eqref{intro:eq-phi}. Furthermore, $\phi$ scatters to the soliton $(Q_{n,k},0)$. 
\end{theorem}

Since the publication of \cite{KLLS15}, there has been much further progress on soliton resolution for equivariant wave maps equations. We particularly highlight the recent breakthrough \cite{JL21}, in which soliton resolution was obtained for two-dimensional equivariant wave maps (on the full spatial domain $\R^2$). \\

Due to Theorem \ref{intro:thm-deterministic}, the deterministic theory of \eqref{intro:eq-phi} is fully understood. In this article, we study \eqref{intro:eq-phi} from a probabilistic perspective, which reveals interesting new aspects. One of the most central directions of research in random dispersive equations, which is inspired by statistical mechanics, concerns the existence and invariance of Gibbs measures. The existence (or construction) of Gibbs measures was initially studied by constructive quantum field theorists (see e.g. the monograph \cite{GJ87}). More recently, it has been studied via stochastic quantization \cite{PW81}, which relies on singular stochastic partial differential equations \cite{AK20,BG18,GH21,MW17,MW20}. 
The invariance of Gibbs measures under dispersive equations was first studied in seminal works of Bourgain \cite{B94} and Zhidkov \cite{Z94}, which treat one-dimensional nonlinear Schrödinger and wave equations, respectively. In recent years, there has also been much progress on invariant Gibbs measures for nonlinear Schrödinger and wave equations in two and three dimensions \cite{B94,B20II,BDNY22,DNY19,DNY20,DNY21,GKO18,OOT21}. 
We emphasize that many of the articles cited above only treat compact domains (such as the periodic box $\mathbb{T}^d$).  Since the exterior equivariant wave maps equation \eqref{intro:eq-phi} is set on the semi-infinite interval $[1,\infty)$, we are interested in the infinite-volume limit of Gibbs measures, which has been considered in \cite{B00,GH21,FO76,MW17,TW+,X14}. 
For a more detailed literature review on the existence and invariance of Gibbs measures, we refer the reader to the introductions of  \cite{GH21} and \cite{B20II,BDNY22}, respectively.  \\

In the following, we study Gibbs measures corresponding to each topological degree $n\geq 0$ and all equivariance-indices $k\geq 1$. Since $\Conn_{n,k}$ from \eqref{intro:eq-connected} is an affine rather than linear space, we first introduce the shift operator $\tau_{n,k}$, which is defined by 
\begin{equation}\label{intro:eq-shift}
\tau_{n,k} \big( \varphi_0, \varphi_1 \big) := \big( Q_{n,k} +\varphi_0, \varphi_1 \big). 
\end{equation}
We then formally define the Gibbs measure $\vec{\mu}_{n,k}$ as the push-forward 
\begin{equation}\label{intro:eq-Gibbs}
\vec{\mu}_{n,k} := \big( \tau_{n,k} \big)_{\#} \vec{\mu}^{\, 0}_{n,k}, 
\end{equation}
where $\vec{\mu}^{\, 0}_{n,k}$ is formally defined by 
\begin{equation}\label{intro:eq-Gibbs-shifted}
`` \, \mathrm{d}\vec{\mu}^{\, 0}_{n,k}( \varphi_0, \varphi_1) = \mathcal{Z}^{-1} \exp \Big( - E_k \big( Q_{n,k} + \varphi_0, \varphi_1 \big) \Big) \mathrm{d}\varphi_0 \mathrm{d}\varphi_1 ".   
\end{equation}
We emphasize that \eqref{intro:eq-Gibbs-shifted} is purely formal, since the energy will later turn out to be infinite on the support of $\vec{\mu}^{\, 0}_{n,k}$ and the infinite-dimensional Lebesgue measure $\mathrm{d}\varphi_0\mathrm{d}\varphi_1$ cannot be defined rigorously. In our main theorem, we prove that the Gibbs measure $\vec{\mu}_{n,k}$ can be constructed rigorously and is invariant under the dynamics of \eqref{intro:eq-phi}. In the following statement, $0<\delta\ll 1$ is a fixed but arbitrary parameter and the weighted Hölder spaces are as in Definition \ref{prelim:def-spaces} below. 

\begin{theorem}\label{intro:thm-main}
For all topological degrees $n\geq 0$ and equivariance-indices  $k\geq 1$, the Gibbs measure $\vec{\mu}_{n,k}$ exists and is supported on the state space
\begin{equation}\label{intro:eq-State}
\begin{aligned}
\State := \bigg\{& \big( \phi_0, \phi_1 \big) \colon 
 r \big( \phi_0- Q_{n,k} \big) \in C_0^{0,1/2-\delta,-1/2-\delta}([1,\infty)), \\ 
 &\,  r \phi_1(r) \in C^{-1,1/2-\delta,-1/2-\delta}([1,\infty)) \bigg\}. 
 \end{aligned}
\end{equation}
Furthermore, the exterior equivariant wave maps equation \eqref{intro:eq-phi} is deterministically globally well-posed on $\State$ and the Gibbs measure $\vec{\mu}_{n,k}$ is invariant under the dynamics. 
\end{theorem}

\begin{remark}
Due to the definition of the weighted Hölder spaces (Definition \ref{prelim:def-spaces}), the initial position $\phi_0$ from Theorem \ref{intro:thm-main} satisfies
\begin{equation*}
\big| \phi_0(r) - Q_{n,k}(r) \big| \lesssim_{\phi_0,n,k,\delta} r^{-1/2+\delta}
\end{equation*} 
for all $r\geq 1$. In contrast, if the initial position $\phi_0$ is as in the connected component from \eqref{intro:eq-connected}, then the radial Sobolev embedding implies that
\begin{equation*}
\big| \phi_0(r) - Q_{n,k}(r) \big| \lesssim_{\phi_0,n,k} r^{-1/2}
\end{equation*} 
for all $r\geq 1$. Thus, while the initial data drawn from $\vec{\mu}_{n,k}$ relaxes to the topological soliton as $r\rightarrow \infty$, the pointwise decay rate is slower than for initial data in the energy class. 
\end{remark}

To the best of our knowledge, Theorem \ref{intro:thm-main} is the first result on the existence and invariance of Gibbs measures which are supported near topological solitons. 
The most difficult part of our main theorem is the existence of the Gibbs measure, which is proven in two steps: In the first step, we study a family of Gaussian measures (Section \ref{section:Gaussian}). The corresponding covariance operators are given by the inverses of the one-dimensional Schrödinger operators
\begin{equation}\label{intro:eq-Schroedinger-operator}
-\partial_r^2 + \frac{k(k+1)}{r^2} \cos\big( 2 Q_{n,k}\big), 
\end{equation}
which involve the topological soliton $Q_{n,k}$. In order to obtain growth and Hölder estimates for the family of Gaussian measures, we rely on Green's function estimates for \eqref{intro:eq-Schroedinger-operator}. 

In the second step, we control the Radon-Nikodym derivatives of the Gibbs measures with respect to the Gaussian measures (Section \ref{section:Gibbs}). Our argument relies on the variational approach of Barashkov and Gubinelli \cite{BG18}, which has also been used in \cite{B20I,OOT21}. In contrast to the argument in \cite{BG18}, however, the objective function in the variational problem is expanded around the drift term rather than the Gaussian term (see Remark \ref{Gibbs:rem-expansion}). \\

In comparison to the construction of the Gibbs measure, the proof of the dynamical aspects of Theorem \ref{intro:thm-main} is rather simple (and all ingredients are essentially contained already in \cite{Z94}). The reason is that, as stated in Theorem \ref{intro:thm-main}, \eqref{intro:eq-phi} is deterministically globally well-posed on the state space $\State$, and thus our argument neither relies on the random structure of the solution (as in \cite{B96,B20II,BDNY22,DNY19,OOT21}) nor Bourgain's globalization argument (as in \cite{B94,B96}). The proof of invariance is slightly technical, since it requires a finite-dimensional approximation of \eqref{intro:eq-phi}, but ultimately follows from similar ingredients as in the deterministic well-posedness theory. \\

Theorem \ref{intro:thm-main} has an interesting consequence for the long-time dynamics of certain solutions of \eqref{intro:eq-phi}, which we record in the following corollary. This corollary involves the linearization of \eqref{intro:eq-phi} around the topological soliton $Q_{n,k}$, which is given by
\begin{equation}\label{intro:eq-phi-lin}
\partial_t^2 \phi_{\textup{lin}} - \partial_r^2 \phi_{\textup{lin}} - \frac{2}{r} \partial_r \phi_{\textup{lin}} + \frac{k(k+1)}{r^2} \cos\big( 2Q_{n,k}\big) \phi_{\textup{lin}} =0.
\end{equation}
In light of Theorem \ref{intro:thm-main}, we are particularly interested in \eqref{intro:eq-phi-lin} with initial data in the linear (rather than affine) state space
\begin{equation*}
  \State^{\textup{lin}} := \Big\{ \big( \phi_0, \phi_1 \big) \colon 
 r  \phi_0  \in C_0^{0,1/2-\delta,-1/2-\delta}([1,\infty)), \, 
  r \phi_1(r) \in C^{-1,1/2-\delta,-1/2-\delta}([1,\infty)) \Big\}.  
\end{equation*}
\begin{corollary}\label{intro:cor-failure-soliton}
Let $n\geq 0$ and let $k\geq 1$. Then, soliton resolution for \eqref{intro:eq-phi} fails $\vec{\mu}_{n,k}$-almost surely. More precisely, there exists an event $A_{n,k}\subseteq \State$, where $\State$ is as in \eqref{intro:eq-State}, such that  $\vec{\mu}_{n,k}(A_{n,k})=1$ and such that the following holds for all $(\phi_0,\phi_1)\in A_{n,k}$: \\ 
Let $\phi$ be the unique global solution of \eqref{intro:eq-phi} with initial data $(\phi_0,\phi_1)$. 
Furthermore, let $\phi_{\textup{lin}}^{+}$ and $\phi_{\textup{lin}}^{-}$  be any solutions of the linearized equation \eqref{intro:eq-phi-lin} with initial data in $\State^{\textup{lin}}$. Then, we have that
\begin{equation}\label{intro:eq-failure-soliton}
\limsup_{t\rightarrow \, \pm \infty} 
\big\| r (\phi - Q_{n,k} - \phi_{\textup{lin}}^\pm )(t,r) \big\|_{(C^{0,1/2-\delta,-1/2-\delta} \times C^{-1,1/2-\delta,-1/2-\delta})([1,2])} >0. 
\end{equation}
\end{corollary}
While \eqref{intro:eq-failure-soliton} is formulated using the same norm as in the definition of the state space $\State$, our argument yields similar conclusions in many other norms (see Remark \ref{proof:rem-generalization}). \\
This corollary is an easy consequence of the properties of the Gibbs measure $\vec{\mu}_{n,k}$ and Poincar\'{e}'s recurrence theorem (see Section \ref{section:proof}). 
The striking aspect of Corollary \ref{intro:cor-failure-soliton} is that soliton resolution fails for certain  $(\phi_0,\phi_1) \in \State$, i.e., the global solution does not decompose into a sum of $Q_{n,k}$ and a linear wave. Since soliton resolution holds for initial data with finite energy (Theorem \ref{intro:thm-deterministic}), this implies that the asymptotic behaviour for random initial data is different from the asymptotic behaviour for smooth initial data.
\begin{remark}
Using the Gaussian measure from Section \ref{section:Gaussian}, we also obtain an invariant Gaussian measure of the linearized wave equation \eqref{intro:eq-phi-lin} which is supported on $\State^{\textup{lin}}$. As a consequence, there exists solutions $\phi_{\textup{lin}}$ of \eqref{intro:eq-phi-lin} with initial data in $\State^{\textup{lin}}$ which do not decay (even locally in space) as time goes to infinity. In light of this, the failure of soliton resolution for initial data in $\State$ may not be too surprising, but it is still interesting that it can be proven. 
\end{remark}

\textbf{Acknowledgements:} The author thanks the anonymous referee and Leonardo Tolomeo for discovering a problem with  Corollary \ref{intro:cor-failure-soliton} in an earlier version of this manuscript. The author thanks  Rowan Killip, Jeremy Marzuola, Igor Rodnianski, and Casey Rodriguez for interesting and helpful discussions. 
The author was partially supported by the NSF under Grant No. DMS-1926686. \\

\section{Preparations} 

In this section, we make necessary preparations for the rest of this article. In Subsection \ref{section:prelim-notation}, we recall basic notation.
In Subsection \ref{section:prelim-analysis} and Subsection \ref{section:prelim-wave}, we recall basic facts from real analysis and the analysis of wave equations, respectively.
In Subsection \ref{section:prelim-change}, we restrict the exterior equivariant wave maps equation \eqref{intro:eq-phi} to finite intervals and introduce a change of variables. Finally, in Subsection \ref{section:prelim-finite-dimensional}, we introduce a finite-dimensional approximation of \eqref{intro:eq-phi}.

\subsection{Notation}\label{section:prelim-notation} 
Let $A,B>0$. We write $A\lesssim B$ if  there exists a constant $C=C(n,k,\delta)>0$ such that $A\leq C B$ is satisfied, where $n$, $k$, and $\delta>0$ are as in Theorem \ref{intro:thm-main}. If the constant $C$ depends on additional parameters, this dependence is indicated through subscripts. For example, if $C$ also depends on $\epsilon>0$, we write $A\lesssim_\epsilon B$. We also write $A\gtrsim B$ if $B\lesssim A$. Finally, we write $A\sim B$ if $A\lesssim B$ and $B\lesssim A$. \\

We further let $R_0=R_0(n,k)\geq 1$ be a sufficiently large radius. In the following, all statements for finite intervals of the form $[1,R]$ will only be made for $R\geq R_0$, which guarantees that the properties from Lemma \ref{prelim:lem-soliton} below are satisfied.

\subsection{Basic facts from analysis} \label{section:prelim-analysis}
In this subsection, we recall a few basic facts from analysis. We first recall the definition of $L^2$-based Sobolev spaces.

\begin{definition}[$L^2$-based Sobolev spaces]\label{prelim:def-Sobolev}
Let $I$ be either the finite interval $[1,R]$, where $R\geq 1$, or the semi-infinite interval $[1,\infty)$. For all smooth, compactly supported $\varphi\colon I \rightarrow \R$, we define the homogeneous norms
\begin{align*}
\big\| \varphi \big\|_{L^2(I)}^2 := \int_I \dr \, |\varphi(r)|^2, 
\quad 
\big\| \varphi \big\|_{\dot{H}^1(I)}^2 := \int_I \dr \, |\partial_r \varphi(r)|^2, 
\quad \text{and} \quad 
\big\| \varphi \big\|_{\dot{H}^2(I)}^2 := \int_I \dr \, |\partial_r^2 \varphi(r)|^2. 
\end{align*}
Furthermore, we define the inhomogeneous norms
\begin{align*}
\big\| \varphi \big\|_{H^1(I)}^2
&:=\big\| \varphi \big\|_{L^2(I)}^2 + \big\| \varphi \big\|_{\dot{H}^1(I)}^2, \\ 
\big\| \varphi \big\|_{H^2(I)}^2
&:=\big\| \varphi \big\|_{L^2(I)}^2 
+ \big\| \varphi \big\|_{\dot{H}^1(I)}^2
+ \big\| \varphi \big\|_{\dot{H}^2(I)}^2.  
\end{align*}
We define the corresponding inhomogeneous function spaces $L^2(I)$, $H^1(I)$, and $H^2(I)$ as the closure of $C^\infty_c(I)$ with respect to the corresponding norms. Furthermore, we define $\dot{H}_0^1(I)$ as the closure of $C^\infty_c(\mathring{I})$, where $\mathring{I}$ is the interior of $I$, with respect to the $\dot{H}^1(I)$-norm. 
\end{definition}

In addition to the $L^2$-based norms, we also work with weighted Hölder norms, which are introduced in the following definition.

\begin{definition}[Weighted Hölder spaces]\label{prelim:def-spaces}
Let $I$ be either the finite interval $[1,R]$, where $R\geq 1$, or the semi-infinite interval $[1,\infty)$, 
let $\kappa \leq 0$, and let $\alpha \in [0,1)$. Then, we
define
    \begin{equation*}
    \big\| \varphi \big\|_{\Czero(I)} 
    := \sup_{r\in I} \big| r^\kappa \varphi(r) \big| 
    + \sup_{\substack{r,\rho \in I\colon \\ r \neq \rho}}
    \bigg| \max(r,\rho)^\kappa \frac{\varphi(r) - \varphi(\rho)}{|r-\rho|^\alpha}\bigg| 
    \end{equation*}
We define the corresponding function space $\Czero(I)$ as the closure of $C^\infty_c(I)$ with respect to the $\Czero(I)$-norm. We also define
\begin{align*}
\Czero_0(I) &:= \Big\{ \varphi \in \Czero(I)\colon \varphi\big|_{\partial I}=0 \Big\}, \\ 
\Czero_{(0)}(I) &:= \Big\{ \varphi \in \Czero(I)\colon \varphi(1)=0 \Big\}. 
\end{align*}
Furthermore, for any locally integrable $\varphi\colon I \rightarrow \R$, we define 
\begin{equation*}
\big\| \varphi \big\|_{\Cmone(I)}
:= \Big\| \int_1^r \drho \, \varphi(\rho) \Big\|_{\Czero(I)}.
\end{equation*}
Finally, we define the corresponding function space $\Cmone(I)$ as the closure of $C^\infty_c(I)$ with respect to the $\Cmone$-norm. 
\end{definition}

We now make a few remarks regarding Definition \ref{prelim:def-spaces}. 
\begin{enumerate}
    \item Since our function spaces (such as $C^{0,\alpha,\kappa}$) are defined as the closure of $C^\infty_c(I)$, our function spaces are slightly different from the usual  Hölder spaces. In particular, all of our function spaces are separable. 
    \item The $\Czero_{(0)}$-spaces, in which the zero Dirichlet boundary condition is only enforced at $r=1$, will be used to compare Gaussian and Gibbs measures defined on different intervals (see e.g. Proposition \ref{Gibbs:prop-Gibbs}). 
    \item The precise definition of the $C^{-1,\alpha,\kappa}$-norm, which contains the integral of $\varphi$, is motivated by d'Alembert's formula (Lemma \ref{prelim:lem-dAlembert}). By using integration by parts, it is easy to see that elements of $C^{-1,\alpha,\kappa}$ are distributions. 
\end{enumerate}

To simplify the notation, we also define the unweighted Hölder norms, i.e., the weighted Hölder norms with $\kappa=0$, by  
\begin{align*}
\big\| \varphi \big\|_{C^{0,\alpha}(I)} := \big\| \varphi \big\|_{C^{0,\alpha,0}(I)} 
\qquad \text{and} \qquad 
\big\| \varphi \big\|_{C^{-1,\alpha}(I)} := \big\| \varphi \big\|_{C^{-1,\alpha,0}(I)}. 
\end{align*}

Finally, we recall a special case of Hardy's inequality.

\begin{lemma}[Hardy's inequality]\label{prelim:lem-Hardy}
For all $R\geq 1$ and all $\zeta \in H^1([1,R])$ satisfying $\zeta(1)=0$, it holds that 
\begin{equation*}
\int_1^R \dr \, \frac{\zeta^2}{r^2} \leq 4 \int_1^R \dr \, (\partial_r \zeta)^2. 
\end{equation*}
\end{lemma}

At the end of this subsection, we introduce extension and restriction operators. 

\begin{definition}[Extension operator]\label{prelim:def-extension}
For any $1\leq R<\infty$ and any $f\colon (1,R) \rightarrow \mathbb{R}$, we define 
$\Ext_R f \colon \R \rightarrow \R$
as the  extension of $f$ which is odd around both $r=1$ and $r=R$. Similarly, for any $f\colon (1,\infty) \rightarrow \R$, we define $\Ext_\infty f \colon \R \rightarrow \R$ as the extension of $f$ which is odd around $r=1$. 
\end{definition}

In the following lemma, we list a few basic properties of the extension operator.

\begin{lemma}[Properties of extension operator]\label{prelim:lem-extension}
For all $1\leq R < \infty$, there exist maps $\ext_R\colon \R \rightarrow [1,R]$ and $\sigma_R \colon \R \rightarrow \{0,1\}$ such that 
\begin{equation*}
\big( \Ext_R f \big) (r) = (-1)^{\sigma_R(r)} f \big( \ext_R(r) \big)
\end{equation*}
for all $f\colon (1,R) \rightarrow \R$. Furthermore, the maps $\ext_R$ and $\sigma_R$ satisfy the following properties: 
\begin{enumerate}[label=(\roman*)]
    \item $\ext_R$ is linear and has slope $\pm 1$ on all intervals of the form $ m \cdot (R-1) + (1,R)$, where $m\in \Z$. 
    \item $\ext_R(r)=r$ for all $r\in (1,R)$. 
    \item $\sigma_R$ is constant on all intervals of the form $ m \cdot (R-1) + (1,R)$, where $m\in \Z$. 
    \item $\sigma_R(r)=0$ for all $r\in (1,R)$. 
\end{enumerate}
With the obvious modifications, the same properties also hold in the semi-infinite case $R=\infty$. 
\end{lemma}

\begin{proof}
The properties follow directly from the definition of the extension operator.
\end{proof}

\begin{definition}[Restriction operators]\label{prelim:def-restriction} Let $1 \leq L \leq R < \infty$. For any smooth $\varphi \colon [1,R] \rightarrow \R$, we define
$\Rest[L][R] \varphi \colon [1,L] \rightarrow \R$ by 
\begin{equation}\label{prelim:eq-rest}
\Rest[L][R] \varphi := \varphi \big|_{[1,L]}. 
\end{equation}
Furthermore, we define $\RestZ[L][R] \varphi \colon [1,L] \rightarrow \R$ by
\begin{equation*}
\RestZ[L][R]\varphi (r) 
= \begin{cases}
\begin{aligned}
\varphi(r)\hspace{22ex} &\qquad \textup{if} \, \, 1\leq r \leq L-1, \\ 
\varphi(L-1) + \big( r- (L-1) \big) \big( \varphi(L)-\varphi(L-1) \big) 
&\qquad \textup{if} \, \, L-1 \leq r \leq L. 
\end{aligned}
\end{cases}
\end{equation*}
Finally, we define
\begin{equation*}
\vecRest[L][R] := \Rest[L][R] \otimes \Rest[L][R] \qquad \text{and} \qquad 
\vecRestZ[L][R] := \RestZ[L][R] \otimes \Rest[L][R]. 
\end{equation*}
\end{definition}
Throughout this article, we will primarily work with the restriction operator $\Rest[L][R]$. However, it can sometimes be important to maintain the zero Dirichlet boundary conditions, and then $\RestZ[L][R]$ will be used. 

\subsection{Wave equations and solitons} \label{section:prelim-wave}

We now recall properties of the one-dimensional  wave equation on the finite interval $[1,R]$ and semi-infinite interval $[1,\infty)$. We first state d'Alembert's formula, which involves the extension operators from Definition \ref{prelim:def-extension}. 

\begin{lemma}[d'Alembert's formula]\label{prelim:lem-dAlembert}
Let $1\leq R < \infty$, let $f\in C^\infty_c((1,R))$, let $g\in C^\infty([1,R])$, and let $h\in C^\infty(\R\times [1,R])$. Then, the unique solution of the initial-boundary value problem 
\begin{equation}
\begin{cases}
\begin{alignedat}{3}
\partial_t^2 u - \partial_r^2 u &= h \hspace{10ex} (&&t,r) \in \R \times (1,R), \\ 
u(t,1)=u(t,R)&=0 &&t\in \R, \\ 
u(0,r)=f(r), \quad u_t(0,r)&=g(r) &&r\in (1,R)
\end{alignedat}
\end{cases}
\end{equation}
is given by 
\begin{equation}
\begin{aligned}
u(t,r) &= \frac{(\Ext_R f)(r+t) + (\Ext_R f)(r-t)}{2} + \frac{1}{2} \int_{r-t}^{r+t} \drho \,  (\Ext_R g)(\rho) \\
&+ \frac{1}{2} \int_{0}^t \ds \int_{r-(t-s)}^{r+(t-s)} \drho \, (\Ext_R h)(s,\rho). 
\end{aligned}
\end{equation}
With the obvious modifications, the same formula also holds in the semi-infinite case $R=\infty$. 
\end{lemma}

In order to simplify the notation, we make the definition
\begin{equation*}
\Duh_R \big[ h \big] 
:=  \frac{1}{2} \int_{0}^t \ds \int_{r-(t-s)}^{r+(t-s)} \drho \, (\Ext_R h)(s,\rho).
\end{equation*}

We now state a precise definition of the topological solitons $Q_{n,k}$, which were informally introduced 
in the introduction.

\begin{definition}[Topological solitons \cite{BCM12,LS13}]\label{prelim:def-soliton}
For all $n\geq 0$ and $k\geq 1$, we define $Q_{n,k}$ as the unique minimizer of 
\begin{equation}
\frac{1}{2} \int_1^\infty \dr \, r^2 \Big( (\partial_r \phi)^2 + \frac{k(k+1)}{r^2} \sin^2\big( \phi \big) \Big) 
\end{equation}
subject to the boundary conditions $\phi(1)=0$ and $\lim_{r\rightarrow \infty} \phi(r) = n\pi$. For notational purposes, it is convenient
to also define $Q_{0,0}(r) \equiv 0$.
\end{definition}

The case $n=k=0$ will only be needed in the definition and analysis of the white noise measure (Definition \ref{Gaussian:def-white-noise}). In the following lemma, we recall
basic properties of the topological solitons.

\begin{lemma}[\protect{Topological solitons \cite{LS13,KLLS15}}]\label{prelim:lem-soliton}
For all $n\geq 0$ and $k\geq 1$, there exists an $\alpha=\alpha_{n,k}\in \mathbb{R}$ such that
\begin{equation*}
\Big| Q_{n,k}(r) - \Big( n \pi - \frac{\alpha}{r^{k+1}} \Big) \Big| \lesssim_{n,k} r^{-3(k+1)}
\end{equation*}
is satisfied for all $r\geq 1$. Furthermore, there exists a constant $c_{n,k}>0$ such that 
\begin{equation*}
\int_1^R \dr \, \psi \, \big( - \partial_r^2 + \tfrac{k(k+1)}{2r^2} \cos\big( 2 Q_{n,k}\big) \big) \,  \psi \geq c_{n,k} \int_1^R \dr \, |\partial_r \psi|^2 
\end{equation*}
for all $R\geq R_0$ and all $\psi \in \dot{H}^1_0([1,R])$. 
\end{lemma} 

\subsection{Restriction to finite intervals and change of variables}\label{section:prelim-change}
In order to rigorously construct the Gibbs measures, we first need to replace the infinite interval in \eqref{intro:eq-phi} by a finite interval. To this end, we let $R\geq R_0$. We then consider
\begin{equation}\label{prelim:eq-phi-R} 
\begin{cases}
\begin{alignedat}{3}
\partial_t^2 \phi_R - \partial_r^2 \phi_R - \frac{2}{r} \partial_r \phi_R + \frac{k(k+1)}{2r^2} \sin\big( 2 \phi_R \big) &= 0 \hspace{20ex} &&(t,r) \in \mathbb{R} \times (1,R), \\ 
\phi_R(t,1) &=0  &&t \in \mathbb{R}, \\ 
\phi_R(t,R) &=Q_{n,k}(R) &&t \in \mathbb{R}, \\ 
\big( \phi_R, \partial_t \phi_R \big)(0,r) &= \big( \phi_{R,0}, \phi_{R,1} \big)(r) &&r \in (1,R). 
\end{alignedat}
\end{cases}
\end{equation}
In \eqref{prelim:eq-phi-R}, we impose the Dirichlet condition $\phi_R(t,R)=Q_{n,k}(R)$, which will guarantee that the limit of $\phi_R$ as $R\rightarrow \infty$ lies in the same homotopy class as $Q_{n,k}$. 
In order for \eqref{prelim:eq-phi-R} to be consistent at $r=R$, we also require that the initial data satisfies $\phi_{0,R}(R)=Q_{n,k}(R)$. The initial-boundary value problem \eqref{prelim:eq-phi-R} has the conserved energy
\begin{equation}\label{prelim:eq-energy}
E_{k,R}\big( \phi_R, \partial_t \phi_R \big) = 
\frac{1}{2} \int_1^R \dr \, r^2 \Big( (\partial_t \phi_R)^2 + (\partial_r \phi_R)^2 + \tfrac{k(k+1)}{r^2} \sin^2 \big( \phi_R \big) \Big). 
\end{equation}
We now introduce a change of variables which separates the topological soliton $Q_{n,k}$ and converts the variable-coefficient operator $\partial_r^2 + 2 r^{-1} \partial_r$ into $\partial_r^2$. To be precise, we write
\begin{equation}\label{prelim:eq-change-of-variables}
\phi_R = Q_{n,k} + r^{-1} \psi_R. 
\end{equation}
The new unknown $\psi_R$ is a solution of the initial-boundary value problem
\begin{equation}\label{prelim:eq-psi}
\begin{cases}
\begin{alignedat}{3}
\partial_t^2 \psi_R - \partial_r^2 \psi_R &= - r^{-1} \Nl\big( r^{-1} \psi_R \big) \hspace{9.5ex} &&(t,r) \in \R \times (1,R), \\ 
\psi_R(t,1) &=0 &&t \in \mathbb{R}, \\ 
\psi_R(t,R) &=0 &&t \in \mathbb{R}, \\ 
\big( \psi_R, \partial_t \psi_R \big)(0,r) &= r \big( \phi_{R,0}-Q_{n,k}, \phi_{R,1} \big)(r) &&r \in (1,R),
\end{alignedat}
\end{cases}
\end{equation}
where 
\begin{equation}\label{prelim:eq-Nl}
\Nl(\varphi) := \frac{k(k+1)}{2} \Big( \sin\big( 2 (Q_{n,k}+ \varphi)\big) - \sin\big( 2 Q_{n,k} \big) \Big). 
\end{equation}
Since the linearization of  
$\sin( 2 (Q_{n,k} +  r^{-1} \psi_R))
- \sin( 2 Q_{n,k}  )$ is $\cos(2Q_{n,k}) (2r^{-1} \psi_R)$, we define a linear operator
\begin{equation}\label{prelim:eq-A}
A_{n,k,R} \colon \mathcal{D}\big( A_{n,k,R}\big) \subseteq L^2([1,R])\rightarrow L^2([1,R]) 
\end{equation}
by 
\begin{equation*}
\mathcal{D}\big( A_{n,k,R}\big) := \big( \dot{H}_0^1 \medcap H^2\big)([1,R])  \quad \text{and}\quad A_{n,k,R}\, \psi_R := \Big( - \partial_r^2 + \frac{k(k+1)}{r^2} \cos\big( 2 Q_{n,k} \big) \Big) \psi_R
\end{equation*}
for all $\psi_R \in \mathcal{D}\big( A_{n,k,R}\big)$. Since $-\partial_r^2$ is self-adjoint and  the multiplication operator corresponding to $\cos(2Q_{n,k})/r^2$ is bounded and self-adjoint, it follows that $A_{n,k,R}$ is self-adjoint. Furthermore, it follows from Lemma \ref{prelim:lem-soliton} that $A_{n,k,R}$ is positive definite. \\

The energy of $\phi_R$ defined as in \eqref{prelim:eq-energy} can also be written in terms of the new unknown $\psi_R$. A direct computation shows that 
\begin{equation}\label{prelim:eq-energy-psi}
\begin{aligned}
E_{k,R}\big( \phi_R, \partial_t \phi_R \big) 
=\, E_{k,R}\big( Q_{n,k}, 0 \big) +
\widetilde{E}_{n,k,R}\big( \psi_R , \partial_t \psi_R \big), 
\end{aligned}
\end{equation}
where 
\begin{equation}
\begin{aligned}
&\widetilde{E}_{n,k,R}\big( \psi_R , \partial_t \psi_R \big) \\
:=& \,   
\frac{1}{2} \int_1^R \dr \Big( (\partial_t \psi_R)^2 + (\partial_r \psi_R)^2  \Big) \\
+&\, \frac{k(k+1)}{2} \int_1^R \dr \Big( \sin\big( Q_{n,k} + r^{-1} \psi_R \big)^2 - \sin\big(Q_{n,k}\big)^2 
- \sin\big( 2Q_{n,k} \big) r^{-1} \psi_R  \Big). 
\end{aligned}
\end{equation}
In the following, the energy $\widetilde{E}_{n,k,R}$ is often decomposed as

\begin{equation}
\begin{aligned}
\widetilde{E}_{n,k,R}\big( \psi_R , \partial_t \psi_R \big)=\frac{1}{2} \int_1^R \dr \Big( (\partial_t \psi_R)^2 + (\partial_r \psi_R)^2 + \frac{k(k+1)}{r^2} \cos\big( 2 Q_{n,k} \big) \psi_R^2 \Big) + V_{n,k,R}(\psi_R),
\end{aligned}
\end{equation}
where the higher-order term $V_{n,k,R}(\psi_R)$ is defined by
\begin{align}
V_{n,k,L}(\psi_R) &:= \frac{k(k+1)}{2} \int_1^L \dr \, \mathscr{V}_{n,k}(\psi_R) \label{prelim:eq-V}, \\ 
\mathscr{V}_{n,k}(\psi_R)&:= 
 \sin^2\big(Q_{n,k}+r^{-1} \psi_R\big) -  \hspace{-0.25ex}  \sin^2\hspace{-0.5ex}\big( Q_{n,k} \big) -  \hspace{-0.25ex}  \sin \hspace{-0.25ex} \big( 2 Q_{n,k} \big) r^{-1} \psi_R  - \hspace{-0.25ex} \cos\big( 2 Q_{n,k} \big) (r^{-1} \psi_R)^2. \label{prelim:eq-scrV}
 \end{align}
We note that the integral density $\scrV_{n,k}$ corresponds to the error in the second-order Taylor expansion of $ \sin^2\big(Q_{n,k}+r^{-1} \psi_R\big)$.

\subsection{Finite-dimensional approximations}\label{section:prelim-finite-dimensional}

In order to prove the invariance of the Gibbs measure, we need to introduce finite-dimensional approximations of the Gibbs measure and dynamics. Our finite-dimensional truncation is based on the eigenfunctions of the differential operator $-\partial_r^2$ with Dirichlet boundary conditions.  We recall that the corresponding orthonormal basis of eigenfunctions is given by 
\begin{equation*}
\bigg\{ \frac{2}{\sqrt{R-1}} \sin \Big( \pi n \frac{r-1}{R-1}\Big) \colon n \geq 1 \bigg\}.     
\end{equation*}
We define $\Proj$ as the $L^2$-orthogonal projection onto the finite-dimensional space 
\begin{equation*}
V_{R,\leq N} := \operatorname{span} \bigg( \bigg\{  \frac{2}{\sqrt{R-1}} \sin \Big( \pi n \frac{r-1}{R-1}\Big) \colon 1 \leq n \leq N \bigg\} \bigg).   
\end{equation*}
We note that $V_{R,\leq N}$ contains functions with frequencies $\lesssim \hspace{-0.5ex} N/R$ (rather than $\lesssim \hspace{-0.5ex} N$). Since the finite-dimensional approximations will only be used for fixed $R\geq 1$, this does not create any problems. In the following lemma, we record a few elementary properties of the projection $\Proj$. 

\begin{lemma}[Properties of $\Proj$]\label{prelim:lem-projection}
Let $R\geq 1$, let $N\geq 1$, and let $\alpha \in [0,1)$. Then, it holds for all $f\in C_0^{0,\alpha}([1,R])$ that
\begin{align}
\Big\| \Proj f \Big\|_{L^2([1,R])} 
&\lesssim R^{1/2} \big\| f \big\|_{L^\infty([1,R])}, 
\label{prelim:eq-projection-1} \\ 
\Big\| \Proj f \Big\|_{L^\infty([1,R])} 
&\lesssim N \big\| f \big\|_{L^\infty([1,R])}, 
\label{prelim:eq-projection-2} \\ 
\Big\| \big( 1- \Proj \big) f \Big\|_{L^2([1,R])} 
&\lesssim R^{1/2} \bigg( \frac{R}{N} \bigg)^\alpha \big\| f \big\|_{C^{0,\alpha}([1,R])}. 
\label{prelim:eq-projection-3}
\end{align}
\end{lemma}

\begin{remark}
The second inequality \eqref{prelim:eq-projection-2} is rather crude and can be improved significantly (using estimates for the Dirichlet kernel). Since it will only be used in soft arguments, however, the precise dependence on $N$ is inessential. 
\end{remark}

\begin{proof}
The first inequality \eqref{prelim:eq-projection-1} follows from the $L^2$-boundedness of $\Proj$ and the embedding $L^\infty \hookrightarrow L^2$. To prove the second inequality \eqref{prelim:eq-projection-2}, we note that
\begin{equation*}
P_N f(r) = \frac{2}{R-1} \sum_{n=1}^N \sin \Big( \pi n \frac{r-1}{R-1} \Big) \int_1^R \drho \, \sin \Big( \pi n \frac{\rho-1}{R-1} \Big) f(\rho). 
\end{equation*}
The desired inequality then follows from the trivial estimate $|\sin(x)|\leq 1$. The third inequality \eqref{prelim:eq-projection-3} with $\alpha=0$ follows from \eqref{prelim:eq-projection-1}. Furthermore, it holds that 
\begin{equation*}
\Big\| \big( 1 - \Proj \big) f \Big\|_{L^2} \lesssim \frac{R}{N} \big\| f^\prime \big\|_{L^2} \lesssim R^{\frac{1}{2}} \frac{R}{N} \big\| f^\prime \big\|_{L^\infty}. 
\end{equation*}
The general case $\alpha \in (0,1)$ of \eqref{prelim:eq-projection-3} then follows by interpolation. 
\end{proof}

Equipped with $\Proj$, we now define the frequency-truncated energy 
\begin{equation*}
\begin{aligned}
&\widetilde{E}_{n,k,R}^{(N)}\big( \psi^{(N)}_R, \partial_t \psi^{(N)}_R \big) \\
:=& \,  
\frac{1}{2} \int_1^R \dr \Big( (\partial_t \psi_R^{(N)})^2 + (\partial_r \psi_R^{(N)})^2 \Big) \\
+&\, \tfrac{k(k+1)}{2} \int_1^R \dr \Big( \sin\big( Q_{n,k} + r^{-1} \Proj \psi_R^{(N)} \big)^2 - \sin\big(Q_{n,k}\big)^2 
- 2 \sin\big( Q_{n,k} \big) r^{-1} \Proj \psi_R^{(N)}  \Big). 
\end{aligned}
\end{equation*}
The energy $\widetilde{E}_{n,k,R}^{(N)}$ leads to the frequency-truncated initial-boundary value problem 
\begin{equation}\label{prelim:eq-psi-R-N}
\begin{cases}
\begin{alignedat}{3}
\partial_t^2 \psi_R^{(N)} - \partial_r^2 \psi_R^{(N)} &= - \Proj \Big( r^{-1} \Nl\big( r^{-1} \Proj \psi_R^{(N)} \big) \Big) \hspace{9.5ex} &&(t,r) \in \R \times (1,R), \\ 
\psi_R^{(N)}(t,1) &=0 &&t \in \mathbb{R}, \\ 
\psi_R^{(N)}(t,R) &=0 &&t \in \mathbb{R}, \\ 
\big( \psi_R^{(N)}, \partial_t \psi_R^{(N)} \big)(0,r) &= r \big( \phi_{R,0}-Q_{n,k}, \phi_{R,1} \big)(r) &&r \in (1,R). 
\end{alignedat}
\end{cases}
\end{equation}

\section{Gaussian measures}\label{section:Gaussian}

As discussed in the introduction, the construction of the Gibbs measures is performed in two steps.
In the first step, which is the subject of this section, we analyze a family of Gaussian measures. Throughout this section, we let $n\geq 0$ and $k\geq 1$ or $n=k=0$ (as in Definition \ref{prelim:def-soliton}). Furthermore, we let $R\geq R_0$, where $R_0$ is as in Section \ref{section:prelim-notation}. 

\begin{definition}[Gaussian measures]\label{Gaussian:def-Gaussian}
We define $\scrg_{n,k,R}$ as the Gaussian measure on $L^2([1,R])$ with covariance operator $A_{n,k,R}^{-1}$, where $A_{n,k,R}$ is as in \eqref{prelim:eq-A}. 
\end{definition}

\begin{remark}\label{Gaussian:rem-representation}
The Gaussian measure $\scrg_{n,k,R}$ is supported on $L^2((1,R))$ since $A_{n,k,R}^{-1}$ is a trace-class operator (for a fixed $R\geq 1$). We can also represent $\scrg_{n,k,R}$ as the law of 
\begin{equation}
\sum_{m=1}^\infty \frac{g_m}{\lambda_m} e_m, 
\end{equation}
where $(e_m)_{m=1}^\infty$ is an orthonormal basis of eigenfunctions of $A_{n,k,R}$ with eigenvalues $(\lambda_m^2)_{m=1}^\infty$ and $(g_m)_{m=1}^\infty$ is a sequence of independent, standard, real-valued Gaussians. 
\end{remark}

In the following proposition, we obtain growth and Hölder estimates for samples from the Gaussian measure $\scrg_{n,k,R}$. 

\begin{proposition}[Gaussian measures]\label{Gaussian:prop-Gaussian}
Let $0<\epsilon\ll 1$ and define $\alpha:=1/2-\epsilon$ and $\kappa:= -1/2-\epsilon$.  Then, it holds for all $p\geq 1$ that 
\begin{equation}\label{Gaussian:eq-Gaussian-e0}
\E_{\scrg_{n,k,R}} \Big[ \big\| \psi \big\|_{\Czero([1,R])}^p \Big]^{1/p} \lesssim_\epsilon \sqrt{p}. 
\end{equation}
Furthermore, it holds for all $r\in [1,R]$ that
\begin{equation}\label{Gaussian:eq-Gaussian-e3}
\E_{\scrg_{n,k,R}} \Big[ \psi(r)^2 \Big] \gtrsim \Big( 1 - \frac{r}{R} \Big) \big( r-1 \big). 
\end{equation}
\end{proposition}

\begin{remark} Proposition \ref{Gaussian:prop-Gaussian} shows  that $\psi(r)$ grows slower than $r^{1/2+\epsilon}$ for all $\epsilon>0$. Thus, the growth rate of $\psi(r)$ is as for Brownian motion, which corresponds to the case $n=k=0$ (in the limit $R\rightarrow \infty$). 
\end{remark}

The proof of Proposition \ref{Gaussian:prop-Gaussian} is postponed until Subsection \ref{section:Gaussian-control} below. While Proposition \ref{Gaussian:prop-Gaussian} yields uniform estimates in $R \geq R_0$, it does not (explicitly) contain the convergence in the infinite-volume limit $R\rightarrow \infty$, which is the subject of the next lemma. 

\begin{lemma}[Infinite-volume limit]\label{Gaussian:lem-infinite}
Let $\alpha:=1/2-\delta$ and let $\kappa:= -1/2-\delta$. Then, there exists a unique Gaussian measure $\scrg_{n,k}$ supported on $\Czero_0([1,\infty))$ which satisfies 
\begin{equation}\label{Gaussian:eq-infinite-e1}
\big( \Rest[L][\infty]\big)_\# \scrg_{n,k} 
= \operatorname{w-lim}\displaylimits_{R\rightarrow \infty} 
\big( \Rest[L][R]\big)_\# \scrg_{n,k,R}
\end{equation}
for all $L\geq 1$. In \eqref{Gaussian:eq-infinite-e1}, the limit refers to the weak limit on $\Czero_{(0)}([1,L])$ (see Definition \ref{prelim:def-spaces} and Definition \ref{appendix:def-weak-convergence}) and $( \Rest[L][\infty])_\# $ and $( \Rest[L][R])_\# $ denote push-forwards. 
\end{lemma}

Just as for Proposition \ref{Gaussian:prop-Gaussian}, the proof of Lemma \ref{Gaussian:lem-infinite} is postponed until Subsection \ref{section:Gaussian-control} below. 
In addition to the Gaussian measures $\scrg_{n,k,R}$ and $\scrg_{n,k}$, which will be used to describe the random initial position, we also need a measure describing the random initial velocity. This measure is the white noise measure, which is defined in the following definition. 

\begin{definition}[White noise measure]\label{Gaussian:def-white-noise}
We define the white noise measure $\scrw_R$ as the push-forward of $\scrg_{0,0,R}$ under the distributional derivative $\partial_r$. 
\end{definition}

Since the potential energy in \eqref{prelim:eq-V} and \eqref{prelim:eq-scrV} only depends on the position but not the velocity, the white noise measure $\scrw_R$ plays a less important role in this article than the Gaussian measures $\scrg_{n,k,R}$. 
In the following corollary, we record the properties of the white noise measure, which easily follow from the corresponding properties of the Gaussian measures. 

\begin{corollary}[White noise measure]\label{Gaussian:cor-white-noise}
Let $0<\epsilon\ll 1$, let $\alpha:=1/2-\epsilon$, and let $\kappa:= -1/2-\epsilon$.  Then, it holds for all $p\geq 1$ that 
\begin{equation}\label{Gaussian:eq-white-noise-e0}
\E_{\scrw_R} \Big[ \big\| \psi \big\|_{\Cmone([1,R])}^p \Big]^{1/p} \lesssim_\epsilon \sqrt{p}. 
\end{equation}
Furthermore, there exists a unique probability measure $\scrw$ supported on $\Cmone([1,\infty))$ which satisfies 
\begin{equation*}
\big( \Rest[L][\infty] \big)_\# \scrw = \operatorname{w-lim}\displaylimits_{R\rightarrow \infty} \big( \Rest[L][R] \big)_\# \scrw_R
\end{equation*}
for all $L\geq 1$. 
\end{corollary}

\begin{proof} This follows directly from the definitions of the $\Czero$ and $\Cmone$-norms, the definition of $\scrw_R$, and Proposition \ref{Gaussian:prop-Gaussian}, and Lemma \ref{Gaussian:lem-infinite}.
\end{proof}

\subsection{The Green's functions} 
In order to prove Proposition \ref{Gaussian:prop-Gaussian}, we require estimates for the Green's function of the operator $A_{n,k,R}$, which is defined in the following definition. 

\begin{definition}[Green's functions]\label{Gaussian:def-Green}
We define $G_{n,k,R}\colon [1,R] \times [1,R]\rightarrow \R$ as the Green's function corresponding to the operator $A_{n,k,R}$, i.e., as the solution of the initial value problem
\begin{equation}
\begin{cases}
\Big( - \partial_r^2 + \frac{k(k+1)}{r^2} \cos\big( 2 Q_{n,k} \big) \Big) G_{k,n,R}(r,\rho) = \delta(r-\rho), \\[1ex]
G_{k,n,R}(1,\rho) = G_{k,n,R}(R,\rho) = 0. 
\end{cases}
\end{equation}
\end{definition}

Since $A_{n,k,R}$ is self-adjoint, the Green's function $G_{n,k,R}$ is symmetric, i.e., satisfies
\begin{equation}\label{Gaussian:eq-Green-symmetry}
G_{n,k,R}(r,\rho) = G_{n,k,R}(\rho,r)
\end{equation}
for all $(r,\rho)\in [1,R]^2$. In the next lemma, we state two representations of the Green's functions. The first representation, which is only available in the special case $n=0$, is explicit. The second representation, which holds for general $n\geq 0$, is an expansion of $G_{n,k,R}$ around $G_{0,k,R}$. 

\begin{lemma}[Representations of the Green's functions]\label{Gaussian:lem-representation}
We have the following two identities. 
\begin{enumerate}[label=(\roman*)]
    \item\label{Gaussian:item-representation-1} The case $n=0$: For all $1\leq r \leq \rho \leq R$, it holds that 
    \begin{equation*}
    G_{0,k,R}(r,\rho) = \frac{1}{1+2\gamma} \frac{R^{1+2\gamma}- \rho^{1+2\gamma}}{R^{1+2\gamma}-1} \Big( \rho^{-\gamma} r^{1+\gamma} - \rho^{-\gamma} r^{-\gamma} \Big),
    \end{equation*}
    where 
    \begin{equation*}
    \gamma= \gamma_k := \sqrt{ \frac{1}{4} + k (k+1)} - 1/2 \geq 0. 
    \end{equation*}
    \item\label{Gaussian:item-representation-2} The case $n\geq 1$:  For all $1\leq r,\rho \leq R$, it holds that 
    \begin{align*}
    G_{n,k,R}(r,\rho) &= G_{0,k,R}(r,\rho) \\
    &+ k (k+1) \int_1^R \du \, G_{0,k,R}(r,u) 
    \, \frac{\cos\big( 2 Q_{n,k}(u)\big)-1}{u^2} \,  G_{0,k,R}(u,\rho) \\
    &+ k^2 (k+1)^2 \int_1^R \du \int_1^R \dv  \bigg[ \, G_{0,k,R}(r,u) 
    \, \frac{\cos\big( 2 Q_{n,k}(u)\big)-1}{u^2} \, G_{n,k,R}(u,v) \\
    &\hspace{2.5ex}\times  
    \, \frac{\cos\big( 2 Q_{n,k}(v)\big)-1}{v^2} \, G_{0,k,R}(v,\rho) \bigg].  
    \end{align*}
\end{enumerate}
\end{lemma} 

\begin{remark}
In the case $n=k=0$, the Green's function is given by 
\begin{equation*}
G_{0,0,R}(r,\rho) = \frac{R-\rho}{R-1} (r-1) 
\end{equation*}
for all $1\leq r \leq \rho \leq R$. This corresponds to a Brownian bridge with starting point $r=1$ and endpoint $r=R$. 
\end{remark}

\begin{proof}
We prove the two identities in \ref{Gaussian:item-representation-1} and \ref{Gaussian:item-representation-2} separately. \\

\emph{Proof of \ref{Gaussian:item-representation-1}:} Since $Q_{0,k}=0$, $G_{0,k,R}$ is a solution of 
\begin{equation}\label{Gaussian:eq-representation-p1}
\Big( - \partial_r^2 + \frac{k(k+1)}{r^2} \Big) G_{0,k,R}(r,\rho) = \delta(r-\rho).     
\end{equation}
The characteristic polynomial equation corresponding to \eqref{Gaussian:eq-representation-p1} is given by 
$- \alpha (\alpha-1)  + k (k+1) = 0$, 
which has the roots 
\begin{equation*}
\bigg\{ \frac{1}{2} + \sqrt{ \frac{1}{4} + k (k+1)}, \frac{1}{2} -\sqrt{ \frac{1}{4} + k (k+1)} \bigg\} = \big\{ 1+ \gamma, - \gamma \big\}. 
\end{equation*}
Thus, the Green's function $G_{0,k,R}$ can be written as 
\begin{equation}\label{Gaussian:eq-representation-p2}
G_{0,k,R}(r,\rho) = 
\begin{cases}
\begin{tabular}{ll}
$a \, r^{1+\gamma} + b \, r^{-\gamma}$ & if $r\leq \rho$,  \\
$c \, r^{1+\gamma} + d \, r^{-\gamma}$ & if $r> \rho$,
\end{tabular}
\end{cases}
\end{equation}
where $a,b,c,d\in \R$ are parameters depending only on $\gamma$, $R$, and $\rho$. In addition to \eqref{Gaussian:eq-representation-p2}, the differential equation \eqref{Gaussian:eq-representation-p1} also implies the two conditions
\begin{equation}\label{Gaussian:eq-representation-p3} 
\lim_{r\uparrow \rho} G_{0,k,R}(r,\rho) = \lim_{r\downarrow \rho} G_{0,k,R}(r,\rho) 
\qquad \text{and} \qquad \lim_{r\uparrow \rho} \partial_r G_{0,k,R}(r,\rho) = \lim_{r\downarrow \rho} \partial_r G_{0,k,R}(r,\rho) +1 .  
\end{equation}
Together with the boundary conditions $G_{0,k,R}(1,\rho)=G_{0,k,R}(R,\rho)=0$, \eqref{Gaussian:eq-representation-p3} yields a linear system for the parameters $a,b,c,d\in \R$, whose solution leads to the desired identity. \\

\emph{Proof of \ref{Gaussian:item-representation-2}:} By using the resolvent identity twice, we obtain that 
\begin{equation}\label{Gaussian:eq-representation-p4}
\begin{aligned}
A_{n,k,R}^{-1} &= A_{0,k,R}^{-1} + A_{n,k,R}^{-1} \big( A_{n,k,R}- A_{0,k,R} \big) A_{0,k,R}^{-1} \\ 
&= A_{0,k,R}^{-1} + A_{0,k,R}^{-1} \big( A_{n,k,R}- A_{0,k,R} \big) A_{0,k,R}^{-1} \\ 
&+A_{0,k,R}^{-1} \big( A_{n,k,R}- A_{0,k,R} \big)  A_{n,k,R}^{-1} \big( A_{n,k,R}- A_{0,k,R} \big) A_{0,k,R}^{-1} 
\end{aligned}
\end{equation}
After converting this operator identity into an identity for the corresponding Green's functions, we obtain the desired identity. 
\end{proof}

In the next lemma, we obtain pointwise and derivative estimates for the Green's functions. These estimates will be the main ingredient in the growth and regularity estimates in Proposition \ref{Gaussian:prop-Gaussian}. 

\begin{lemma}[Growth and derivative estimates for the Green's functions]\label{Gaussian:lem-Green-estimate}
It holds for all $1\leq r,\rho\leq R$ that 
\begin{align}
\big| G_{k,n,R}(r,\rho) \big| &\lesssim \min(r,\rho), \label{Gaussian:eq-Green-estimate-1}\\ 
\big| \partial_r G_{k,n,R}(r,\rho) \big| &\lesssim 1, \label{Gaussian:eq-Green-estimate-2} \\
\big|  \partial_\rho G_{k,n,R}(r,\rho) \big| &\lesssim 1. \label{Gaussian:eq-Green-estimate-3}
\end{align}
\end{lemma}

\begin{proof}
We split the proof into two steps. In the first step, we treat the special case $n=0$, which uses  Lemma \ref{Gaussian:lem-representation}.\ref{Gaussian:item-representation-1}. In the second step, we then treat the general case $n\geq 1$, which uses the resolvent identity from  Lemma \ref{Gaussian:lem-representation}.\ref{Gaussian:item-representation-2}.\\

\emph{Step I: The special case $n=0$.} We separately prove the three estimates \eqref{Gaussian:eq-Green-estimate-1}, \eqref{Gaussian:eq-Green-estimate-2}, and \eqref{Gaussian:eq-Green-estimate-3}. Due to the symmetry of the Green's functions, it suffices to treat the case $1\leq r \leq \rho\leq R$.
Using Lemma \ref{Gaussian:lem-representation}.\ref{Gaussian:item-representation-1}, we obtain the pointwise estimate
\begin{align*}
G_{0,k,R}(r,\rho) &= \frac{1}{1+2\gamma} \frac{R^{1+2\gamma}- \rho^{1+2\gamma}}{R^{1+2\gamma}-1} \Big( \rho^{-\gamma} r^{1+\gamma} - \rho^{-\gamma} r^{-\gamma} \Big) 
\leq \frac{1}{1+2\gamma} \cdot 1 \cdot r \lesssim r. 
\end{align*}
Similarly, we obtain the $\partial_r$-estimate 
\begin{align*}
\big| \partial_r G_{0,k,R}(r,\rho) \big| 
\leq \frac{1}{1+2\gamma} \frac{R^{1+2\gamma}- \rho^{1+2\gamma}}{R^{1+2\gamma}-1} \Big( (1+\gamma) \Big( \frac{r}{\rho}\Big)^\gamma + \gamma \rho^{-\gamma} r^{-1-\gamma} \Big) 
\lesssim 1. 
\end{align*}
In order to obtain the $\partial_\rho$-estimate, we first decompose
\begin{align}
\partial_\rho G_{k,0,R}(r,\rho) &= 
\frac{1}{1+2\gamma}  \partial_\rho \bigg( \frac{R^{1+2\gamma}- \rho^{1+2\gamma}}{R^{1+2\gamma}-1} \bigg)  \Big( \rho^{-\gamma} r^{1+\gamma} - \rho^{-\gamma} r^{-\gamma} \Big) \label{Gaussian:eq-Green-estimate-p1} \\ 
&+ \frac{1}{1+2\gamma}   \frac{R^{1+2\gamma}- \rho^{1+2\gamma}}{R^{1+2\gamma}-1} \partial_\rho  \Big( \rho^{-\gamma} r^{1+\gamma} - \rho^{-\gamma} r^{-\gamma} \Big).  \label{Gaussian:eq-Green-estimate-p2}
\end{align}
For the first summand \eqref{Gaussian:eq-Green-estimate-p1} and second summand \eqref{Gaussian:eq-Green-estimate-p2}, we then have that
\begin{equation*}
\big| \eqref{Gaussian:eq-Green-estimate-p1} \big| \lesssim \frac{\rho^{2\gamma}}{R^{1+2\gamma}} r \lesssim 1 \qquad \text{and} \qquad \big| \eqref{Gaussian:eq-Green-estimate-p2} \big| \lesssim 
\Big( \frac{r}{\rho} \Big)^{1+\gamma} +1 \lesssim 1.
\end{equation*}
This completes the proof of the three estimates \eqref{Gaussian:eq-Green-estimate-1}, \eqref{Gaussian:eq-Green-estimate-2}, and \eqref{Gaussian:eq-Green-estimate-3} in the special case $n=0$. \\ 

\emph{Step II: The general case $n\geq 1$.} We first prove for all $1\leq r \leq \rho \leq R$ that 
\begin{equation}\label{Gaussian:eq-Green-estimate-p3} 
\big| G_{n,k,R}(r,\rho) \big| \lesssim \sqrt{r\rho}.
\end{equation}
To this end, we use Cauchy-Schwarz, which implies
\begin{align*}
\Big| G_{n,k,R}(r,\rho)\Big|^2 &= \Big| \big \langle \delta_r, A_{n,k,R}^{-1} \delta_\rho \big\rangle_{L^2} \Big|^2 
\leq \big\| A_{n,k,R}^{-1/2}\delta_r \big\|_{L^2}^2 \cdot \big\| A_{n,k,R}^{-1/2} \delta_\rho \big\|_{L^2}^2 \\
&= \big\langle \delta_r , A_{n,k,R}^{-1} \delta_r \big\rangle  \cdot 
\big\langle \delta_\rho , A_{n,k,R}^{-1} \delta_\rho \big\rangle 
= G_{n,k,R}(r,r) G_{n,k,R}(\rho,\rho).  
\end{align*}
Thus, it suffices to treat the case $r=\rho$. Due to Lemma \ref{prelim:lem-soliton}, there exists a positive constant $c_{n,k}>0$ such that  
$A_{n,k,R} \succeq  c_{n,k} A_{0,k,R}$. 
Due to the operator monotonicity of the operator inverse, it follows that 
$A_{n,k,R}^{-1} \preceq c_{n,k}^{-1} A_{0,k,R}^{-1}$. 
At the level of the Green's function, it then follows that 
\begin{equation*}
G_{n,k,R}(r,r) \leq c_{n,k}^{-1} G_{0,k,R}(r,r) \lesssim r. 
\end{equation*}
This completes the proof of \eqref{Gaussian:eq-Green-estimate-p3}. \\

We now prove the desired estimates \eqref{Gaussian:eq-Green-estimate-1}, \eqref{Gaussian:eq-Green-estimate-2}, and \eqref{Gaussian:eq-Green-estimate-3}. Due to the symmetry of $G_{n,k,R}$, it suffices to prove \eqref{Gaussian:eq-Green-estimate-1} and \eqref{Gaussian:eq-Green-estimate-2}. We now recall the resolvent identity\footnote{Since we are using symmetry to only estimate the $r$-derivative, it would have been sufficient to use a first-order rather than second-order expansion in \eqref{Gaussian:eq-representation-p4}. For expository purposes, however, we prefer to work with the second-order expansion. As a result of the second-order expansion, both the $r$ and $\rho$-derivatives of  \eqref{Gaussian:eq-Green-estimate-p4}-\eqref{Gaussian:eq-Green-estimate-p6} can be estimated.} from Lemma \ref{Gaussian:lem-representation}.\ref{Gaussian:item-representation-2}, which yields
\begin{align}
    G_{n,k,R}(r,\rho) &= G_{0,k,R}(r,\rho) \label{Gaussian:eq-Green-estimate-p4}\\
    &+ k (k+1) \int_1^R \du \, G_{0,k,R}(r,u) 
    \, \frac{\cos\big( 2 Q_{n,k}(u)\big)-1}{u^2} \,  G_{0,k,R}(u,\rho)  \label{Gaussian:eq-Green-estimate-p5} \\
    &+ k^2 (k+1)^2 \int_1^R \du \int_1^R \dv  \bigg[ \, G_{0,k,R}(r,u) 
    \, \frac{\cos\big( 2 Q_{n,k}(u)\big)-1}{u^2} \, G_{n,k,R}(u,v) \label{Gaussian:eq-Green-estimate-p6}\\
    &\hspace{2.5ex}\times  
    \, \frac{\cos\big( 2 Q_{n,k}(v)\big)-1}{v^2} \, G_{0,k,R}(v,\rho) \bigg]. \notag  
    \end{align}
We emphasize that in all three terms \eqref{Gaussian:eq-Green-estimate-p4}, \eqref{Gaussian:eq-Green-estimate-p5}, and \eqref{Gaussian:eq-Green-estimate-p6}, the $r$ and $\rho$-variables only enter as arguments of $G_{0,k,R}$, which is crucial for upgrading the pointwise estimate \eqref{Gaussian:eq-Green-estimate-p3} to derivative estimates. 
It suffices to prove the pointwise estimate \eqref{Gaussian:eq-Green-estimate-1} and derivative estimate \eqref{Gaussian:eq-Green-estimate-2} separately for the three summands  \eqref{Gaussian:eq-Green-estimate-p4}, \eqref{Gaussian:eq-Green-estimate-p5}, and \eqref{Gaussian:eq-Green-estimate-p6}. \\

For the first summand \eqref{Gaussian:eq-Green-estimate-p4}, the estimates \eqref{Gaussian:eq-Green-estimate-1} and \eqref{Gaussian:eq-Green-estimate-2} have already been proven in the first step. For the second summand, Lemma \ref{prelim:lem-soliton} implies that 
\begin{align*}
\big| \eqref{Gaussian:eq-Green-estimate-p5} \big|
\lesssim \int_1^R \du \, \min(r,u) \times u^{-6} \times \min(u,\rho) \lesssim   \int_1^R \du \, u^{-4} \lesssim 1 \lesssim \min(r,\rho). 
\end{align*}
Similarly, we have that 
\begin{align*}
\big| \partial_r \eqref{Gaussian:eq-Green-estimate-p5} \big| 
\lesssim \int_1^R \du \,  \big| \partial_r G_{0,k,R}(r,u) \big| u^{-6} \min(u,\rho) 
\lesssim \int_1^R \du \, u^{-5} \lesssim 1. 
\end{align*}
It remains to treat the third summand \eqref{Gaussian:eq-Green-estimate-p6}. Using Lemma \ref{prelim:lem-soliton} and \eqref{Gaussian:eq-Green-estimate-p3}, it holds that 
\begin{align*}
    \big| \eqref{Gaussian:eq-Green-estimate-p6} \big|
    \lesssim \int_1^R \du \int_1^R \dv \, \min(r,u) \, u^{-6} \, \sqrt{uv} \, v^{-6} \, \min(v,\rho)  \lesssim    \int_1^R \du \int_1^R \dv \, u^{-9/2} v^{-9/2} \lesssim 1. 
\end{align*}
Similarly, we have that 
\begin{align*}
     \big| \partial_r \eqref{Gaussian:eq-Green-estimate-p6} \big|
    &\lesssim \int_1^R \du \int_1^R \dv \, \big| \partial_r G_{0,k,R}(r,u) \big| 
    u^{-6} \, \sqrt{uv} \, v^{-6} \, \min(v,\rho) \\
    &\lesssim \int_1^R \du \int_1^R \dv \, u^{-5} v^{-9/2} \lesssim 1. \qedhere
\end{align*}
\end{proof}

In the next lemma, we obtain a lower bound for the diagonal of the Green's function, which essentially matches the upper bound from Lemma \ref{Gaussian:lem-Green-estimate}.

\begin{lemma}[Lower bounds]\label{Gaussian:lem-lower-bound}
It holds for all $1\leq r \leq R$ that 
\begin{equation}
G_{n,k,R}(r,r) \gtrsim \Big( 1 - \frac{r}{R} \Big) \big( r-1 \big). 
\end{equation}
\end{lemma}

\begin{proof}
Using the trivial estimate $\cos(2Q_{n,k})\leq 1$, it holds that 
\begin{equation*}
-\partial_r^2 + \frac{k(k+1)}{r^2} \cos\big( 2 Q_{n,k} \big)  \preceq -\partial_r^2 + \frac{k(k+1)}{r^2}. 
\end{equation*}
Due to the operator monotonicity of the inverse, it follows that $A_{n,k,R}^{-1} \succeq A_{0,k,R}^{-1}$. 
As a result, it follows for all $r\in [1,R]$ that 
\begin{equation*}
G_{n,k,R}(r,r) \geq G_{0,k,R}(r,r). 
\end{equation*}
Using Lemma \ref{Gaussian:lem-representation} and $R\geq R_0 \gg 1$, we obtain that
\begin{equation*}
G_{0,k,R}(r,r) = \frac{1}{1+2\gamma} \frac{R^{1+2\gamma}- r^{1+2\gamma}}{R^{1+2\gamma}-1} (r-r^{-2\gamma}) 
\gtrsim   \frac{R^{1+2\gamma}- r R^{2\gamma}}{R^{1+2\gamma}-1} (r-1) \gtrsim  \Big( 1 - \frac{r}{R} \Big) \big( r-1 \big),
\end{equation*}
which yields the desired estimate. 
\end{proof}

We now state and proof our last estimate for the Green's functions, which concerns the limit as $R\rightarrow \infty$. 

\begin{lemma}[Convergence of Green's functions]\label{Gaussian:lem-Green-convergence}
For all $L\geq 1$, it holds that 
\begin{equation}\label{Gaussian:eq-Green-convergence}
\lim_{R,R^\prime \rightarrow \infty} \int_1^L \dr \int_1^L \drho \, \big| G_{n,k,R}(r,\rho) - G_{n,k,R^\prime}(r,\rho) \big| =0. 
\end{equation}
\end{lemma}

Our argument is based on a weighted energy estimate. 

\begin{proof}
For expository purposes, we separate the proof into four steps. \\

\emph{Step 1: Setup.} Due to the limit and symmetry in $R$ and $R^\prime$, we may assume that $R^\prime \geq R \gg L$. We further fix $\rho\in [1,L]$ and let $0< \eta \ll 1$ remain to be chosen. We then define the weighted difference $w=w_{n,k,R,R^\prime,\rho}\colon [1,R]\rightarrow \R$ by 
\begin{equation}\label{Gaussian:eq-Green-convergence-p0}
w(r) := \Big( \frac{R}{r}\Big)^\eta \Big( G_{n,k,R^\prime}(r,\rho)- G_{n,k,R}(r,\rho) \Big). 
\end{equation}
A straightforward calculation shows that $w$ solves the initial-boundary value problem 
\begin{equation}\label{Gaussian:eq-Green-convergence-p1}
\begin{cases}
\Big( - \partial_r^2 + \frac{k(k+1)}{r^2} \cos\big( 2 Q_{n,k} \big) \Big) w(r) = - 2 \eta r^{-1} \partial_r w(r) + \eta (1-\eta) r^{-2} w(r), \\[1ex]
w(1) =0, \, w(R) = G_{k,n,R^\prime}(R,\rho). 
\end{cases}
\end{equation}
From Lemma \ref{Gaussian:lem-Green-estimate}, it also follows that
\begin{equation}\label{Gaussian:eq-Green-convergence-p2}
|w(R)| \lesssim |\rho| \lesssim L \qquad \text{and} \qquad |\partial_r w(R) | \lesssim 1. 
\end{equation}
\emph{Step 2: Weighted energy estimate.} In the second step, we prove the energy estimate
\begin{equation}\label{Gaussian:eq-Green-convergence-p3}
\int_1^R \dr \Big( (\partial_r w)^2 + \frac{k(k+1)}{r^2} \cos\big( 2 Q_{n,k}\big) w^2 \Big) 
\lesssim L + \eta \int_1^R \dr \, (\partial_r w)^2. 
\end{equation}
In order to prove \eqref{Gaussian:eq-Green-convergence-p3}, we multiply the ordinary differential equation in \eqref{Gaussian:eq-Green-convergence-p1} with $w$ and integrate by parts, which yields the identity 
\begin{equation}\label{Gaussian:eq-Green-convergence-p4}
\begin{aligned}
&\int_1^R \dr \Big( (\partial_r w)^2 + \frac{k(k+1)}{r^2} \cos\big( 2 Q_{n,k}\big) w^2 \Big)  \\
=&\, w(r) \partial_r w(r)\Big|_{r=1}^R - 2 \eta \int_1^R \dr \, r^{-1} w \partial_r w + \eta (1-\eta) \int_1^R \dr \, r^{-2} w^2. 
\end{aligned}
\end{equation}
Using the boundary conditions in \eqref{Gaussian:eq-Green-convergence-p1} and the estimates in \eqref{Gaussian:eq-Green-convergence-p2}, the boundary terms in \eqref{Gaussian:eq-Green-convergence-p4} can be estimated by 
\begin{equation*}
\bigg| w(r) \partial_r w(r)\Big|_{r=1}^R \bigg| 
= \bigg| w(R) \partial_r w(R) \bigg| \lesssim L. 
\end{equation*}
The second and third term in \eqref{Gaussian:eq-Green-convergence-p4} can be estimated using Cauchy-Schwarz and Hardy's inequality. This completes the proof of \eqref{Gaussian:eq-Green-convergence-p3}. \\

\emph{Step 3: Positive definiteness.} In this step, we show that 
\begin{equation}\label{Gaussian:eq-Green-convergence-p5}
\begin{aligned}
\int_1^R \dr \Big( (\partial_r w)^2 + \frac{k(k+1)}{r^2} \cos\big( 2 Q_{n,k}\big) w^2 \Big) + L  \gtrsim \int_1^R \dr \, \Big( (\partial_r w)^2 + r^{-2} w^2 \Big). 
\end{aligned}
\end{equation}
In order to utilize Lemma \ref{prelim:lem-soliton}, we need to replace $w$ with an element of $H_0^1([1,R])$. To this end, we let $\chi \colon \R \rightarrow [0,1]$ be a smooth cut-off function satisfying $\chi|_{[-1/4,1/4]}=1$ and $\chi|_{\R\backslash [-1/2,1/2]}=0$. We then define $\widetilde{w}$ by 
\begin{equation}\label{Gaussian:eq-Green-convergence-p6}
\widetilde{w}(r):= w(r) - \chi\Big( \frac{r-R}{R} \Big) w(R)
\end{equation}
and note that $\widetilde{w}$ satisfies the boundary conditions $\widetilde{w}(1)=\widetilde{w}(R)=0$. Using \eqref{Gaussian:eq-Green-convergence-p2}, it follows that 
\begin{equation}\label{Gaussian:eq-Green-convergence-p7}
\begin{aligned}
&\int_1^R \dr \Big( (\partial_r \widetilde{w} - \partial_r w)^2 + r^{-2} (\widetilde{w}-w)^2 \Big) \\ 
\lesssim& \, \bigg( \frac{1}{R^2} \int_1^R \dr \, \chi^\prime\Big( \frac{r-R}{R} \Big)^2 + \int_{R/2}^R \dr \, r^{-2} \bigg) w(R)^2  \\
\lesssim& \, R^{-1} w(R)^2 \lesssim R^{-1} L^2 \lesssim L. 
\end{aligned}
\end{equation}
The desired estimate \eqref{Gaussian:eq-Green-convergence-p5} can now be derived from Lemma \ref{prelim:lem-soliton} (applied to $\widetilde{w}$) and \eqref{Gaussian:eq-Green-convergence-p7}. \\

\emph{Step 4: Conclusion.} Provided that $0<\eta \ll 1$ is sufficiently small, \eqref{Gaussian:eq-Green-convergence-p3} and \eqref{Gaussian:eq-Green-convergence-p5} yield
\begin{equation}\label{Gaussian:eq-Green-convergence-p8}
\int_1^R \dr \, r^{-2} w^2 \leq \int_1^R \dr \, \Big( (\partial_r w)^2 + r^{-2} w^2 \Big) \lesssim L. 
\end{equation}
By restricting the domain of integration to $[1,L]$, inserting the definition of $w$ from \eqref{Gaussian:eq-Green-convergence-p0}, and recalling that $\rho\in[1,L]$ is arbitrary, it follows that
\begin{equation}\label{Gaussian:eq-Green-convergence-p9}
\sup_{\rho\in [1,L]} \int_1^L \dr \, \big| G_{n,k,R^\prime}(r,\rho)- G_{n,k,R}(r,\rho)\big|^2 \lesssim L^{1+\eta} R^{-\eta}. 
\end{equation}
Together with Hölder's inequality, this implies the desired estimate \eqref{Gaussian:eq-Green-convergence}. 
\end{proof}

\subsection{Control of Gaussian measure}\label{section:Gaussian-control}  We first recall a special case of Mercer's theorem (cf. \cite[Section III.5.4]{CH53}), which allows us to utilize our 
Green's function estimates (Lemma \ref{Gaussian:lem-Green-estimate}). 

\begin{lemma}\label{Gaussian:lem-Mercer}
For all $1\leq r,\rho\leq R$,  it holds that 
\begin{equation*}
\E_{\scrg_{n,k,R}} \Big[ \psi(r) \psi(\rho) \Big] = G_{n,k,R}(r,\rho). 
\end{equation*}
\end{lemma}

\begin{proof}
We rely on the representation of the Gaussian measure $\scrg_{n,k,R}$ from Remark \ref{Gaussian:rem-representation}. From this representation, it follows that
\begin{equation*}
    \E_{\scrg_{n,k,R}} \Big[ \psi(r) \psi(\rho) \Big] = \sum_{m=1}^\infty \frac{1}{\lambda_m^2} e_m(r) e_m(\rho) = G_{n,k,R}(r,\rho). \qedhere
\end{equation*}
\end{proof}

Equipped with Lemma \ref{Gaussian:lem-Mercer}, we now have all ingredients for our proof of Proposition \ref{Gaussian:prop-Gaussian}. 

\begin{proof}[Proof of Proposition \ref{Gaussian:prop-Gaussian}]
Using the definition of the $\Czero([1,R])$-norm, the estimate \eqref{Gaussian:eq-Gaussian-e0} can be reduced\footnote{In fact, \eqref{Gaussian:eq-Gaussian-e2} is stronger than the required estimate, since it contains $\max(r,\rho)^{-\epsilon}$ instead of $\max(r,\rho)^{-1/2-\epsilon}$.} to the two estimates 
\begin{align}
\E_{\scrg_{n,k,R}} \bigg[ \sup_{1\leq r\leq R} \bigg(\frac{|\psi(r)|}{r^{1/2+\epsilon}}\bigg)^p \bigg]^{1/p} &\lesssim_{\epsilon} \sqrt{p} \label{Gaussian:eq-Gaussian-e1},  \\ 
\E_{\scrg_{n,k,R}} \bigg[ \sup_{\substack{1\leq r,\rho \leq R \colon \\ r\neq s}} \bigg(   \frac{|\psi(r)-\psi(\rho)|}{\max(r,\rho)^\epsilon \cdot |r-\rho|^{(1-\epsilon)/2}}\bigg)^p \bigg]^{1/p} &\lesssim_{\epsilon} \sqrt{p} \label{Gaussian:eq-Gaussian-e2}. 
\end{align}
It suffices to treat the case $p\geq 10 \epsilon^{-1}$, since the case $p\leq 10 \epsilon^{-1}$ then follows from Hölder's inequality. 
The following argument is a combination of Mercer's theorem (Lemma \ref{Gaussian:lem-Mercer}), the Green's function estimate (Lemma \ref{Gaussian:lem-Green-estimate}), and Kolmogorov's continuity theorem (Lemma \ref{prelim:lem-kolmogorov}). \\

Using Mercer's theorem (Lemma \ref{Gaussian:lem-Mercer}) and the Green's function estimate (Lemma \ref{Gaussian:lem-Green-estimate}), we obtain for all $1\leq r,\rho \leq R$ that 
\begin{align*}
\E_{\scrg_{n,k,R}} \Big[ |\psi(r) - \psi(\rho)|^2 \Big] 
&= \E_{\scrg_{n,k,R}} \Big[ \psi(r)^2 \Big] - 2 \E_{\scrg_{n,k,R}} \Big[ \psi(r) \psi(\rho) \Big] + \E_{\scrg_{n,k,R}} \Big[ \psi(\rho)^2 \Big] \\
&= G_{n,k,R}(r,r) - 2 G_{n,k,R}(r,\rho) + G_{n,k,R}(\rho,\rho) \\ 
&\lesssim \Big( \max_{1\leq u \leq R} |\partial_r G_{n,k,R}(u,\rho)| 
+ \max_{1\leq u \leq R} |\partial_\rho G_{n,k,R}(r,u)| \Big) |r-\rho| \\
&\lesssim |r-\rho|. 
\end{align*}
Using Gaussian hypercontractivity (Lemma \ref{prelim:lem-properties-Gaussian}), we obtain for all $p\geq 1$ that 
\begin{equation*}
\E_{\scrg_{n,k,R}} \Big[ |\psi(r) - \psi(\rho)|^p \Big]^{1/p} \lesssim \sqrt{p} |r-\rho|.  
\end{equation*}
We now let $1\leq L\leq R$. Using Kolmogorov's continuity theorem (Lemma \ref{prelim:lem-kolmogorov}) with $\alpha=1/2-1/p$ and $\beta=(1-\epsilon)/2$ and using that $p\geq 10\epsilon^{-1}$, we obtain that 
\begin{align*}
&\E_{\scrg_{n,k,R}} \bigg[ \sup_{\substack{1\leq r,\rho \leq R \colon \\ r\neq s, \\ 
\max(r,\rho) \in [L/4,L]}} \bigg(   \frac{|\psi(r)-\psi(\rho)|}{\max(r,\rho)^\epsilon \cdot |r-\rho|^{(1-\epsilon)/2}}\bigg)^p \bigg]^{1/p} \\
\lesssim_\epsilon& L^{-\epsilon} \, \E_{\scrg_{n,k,R}} \bigg[ \sup_{\substack{1\leq r,\rho \leq L \colon \\ r\neq s}} \bigg(   \frac{|\psi(r)-\psi(\rho)|}{ |r-\rho|^{(1-\epsilon)/2}}\bigg)^p \bigg]^{1/p} \\
\lesssim_\epsilon& \sqrt{p} L^{-\epsilon} L^{\frac{1}{p}+\frac{1}{2}-\frac{1-\epsilon}{2}} \lesssim \sqrt{p} L^{-\epsilon/4}. 
\end{align*}

After summing over all dyadic $L\in [1,R]$, this yields the Hölder estimate \eqref{Gaussian:eq-Gaussian-e2}. The growth estimate \eqref{Gaussian:eq-Gaussian-e1} then directly follows from the boundary condition $\psi(1)=0$ and the Hölder estimate \eqref{Gaussian:eq-Gaussian-e2}. 
It now only remains to prove the lower bound \eqref{Gaussian:eq-Gaussian-e3}. Using Lemma \ref{Gaussian:lem-Mercer}, it holds that 
\begin{equation*}
\E_{\scrg_{n,k,R}} \Big[ \psi(r)^2 \Big] = G_{n,k,R}(r,r). 
\end{equation*}
Using Lemma \ref{Gaussian:lem-lower-bound}, we directly obtain the desired estimate.
\end{proof} 

It remains to prove Lemma \ref{Gaussian:lem-infinite}, which concerns the infinite-volume limit of the Gaussian measures. 

\begin{proof}[Proof of Lemma \ref{Gaussian:lem-infinite}]
It suffices to prove the existence of the weak limit 
\begin{equation}\label{Gaussian:eq-infinite-p1} 
\scrg_{n,k,(L)} := 
 \operatorname{w-lim}\displaylimits_{R\rightarrow \infty} 
\big( \Rest[L][R]\big)_\# \scrg_{n,k,R}
\end{equation} 
on $\Czero_{(0)}([1,L])$ for all $L\geq 1$. Indeed, once \eqref{Gaussian:eq-infinite-p1} has been established, the Gaussian measure $\scrg_{n,k}$ can be constructed from $(\scrg_{n,k,(L)})_{L\geq 1}$ via Kolmogorov's extension theorem. From Proposition \ref{Gaussian:prop-Gaussian}, it follows that the Gaussian measures $((\Rest[L][R])_\# \scrg_{n,k,R})_{R\geq 1}$ are tight on $\Czero_{(0)}([1,L])$. Due to Prokhorov's theorem, it therefore only remains to establish the uniqueness of weak subsequential limits of $((\Rest[L][R])_\# \scrg_{n,k,R})_{R\geq 1}$. For any $\xi_L \in C^\infty_c((1,L))$, the law of the random variable
\begin{equation}\label{Gaussian:eq-infinite-p2}
\psi_L \in \Czero_{(0)}([1,L]) \mapsto \int_1^L \dr \,  \xi_L(r) \psi_L(r) 
\end{equation}
with respect to the Gaussian measure $(\Rest[L][R])_\# \scrg_{n,k,R}$ is a normal distribution with mean zero and variance
\begin{equation}\label{Gaussian:eq-infinite-p3}
\int_1^L \dr \int_1^L \drho \, G_{n,k,R}(r,\rho) \xi_L(r) \xi_L(\rho). 
\end{equation}
In order to prove the uniqueness of weak subsequential limits, it therefore suffices to prove the convergence of \eqref{Gaussian:eq-infinite-p2} as $R\rightarrow \infty$. This follows directly from the convergence of the Green's functions $G_{n,k,R}$ as stated in Lemma \ref{Gaussian:lem-Green-convergence}. 
\end{proof}

\section{Existence of the Gibbs measures}\label{section:Gibbs}

In this section, we construct the Gibbs measures. As in Section \ref{section:Gaussian}, we continue to work with the unknown $\psi_R$ from \eqref{prelim:eq-change-of-variables}. In order to distinguish between the Gibbs measures in $\phi_R$ and $\psi_R$, we denote the corresponding Gibbs measures by $\vec{\mu}_{n,k,R}$ and $\vec{\nu}_{n,k,R}$, respectively. Throughout this section, we primarily work with $\vec{\nu}_{n,k,R}$, and later convert our result to $\vec{\mu}_{n,k,R}$. \\

In the first definition of this section, we introduce the Gibbs measures corresponding to the frequency-truncated $k$-equivariant wave maps equation \eqref{prelim:eq-psi-R-N}.

\begin{definition}[Frequency-truncated Gibbs measures]
Let $n\geq 0$, $k\geq 1$, $R\geq R_0$, and $N\geq 1$. Then, we define 
\begin{equation}\label{Gibbs:eq-Gibbs-def}
\begin{aligned}
\nu^{(N)}_{n,k,R} := \big( \mathcal{Z}_{n,k,R}^{(N)}\big)^{-1} \exp \big( - V^{(N)}_{n,k,R} \big) \scrg_{n,k,R}. 
\end{aligned}
\end{equation}
In $\eqref{Gibbs:eq-Gibbs-def}$,  $\mathcal{Z}^{(N)}_{n,k,R}>0$ is a normalization constant,  $\scrg_{n,k,R}$ is as in Definition \ref{Gaussian:def-Gaussian},
\begin{align}
V^{(N)}_{n,k,R}(\psi_R) &:= \frac{k(k+1)}{2} \int_1^L \dr \, \mathscr{V}^{(N)}_{n,k}(\psi_R) \label{Gibbs:eq-V-N}, \\ 
\text{and} \quad \mathscr{V}^{(N)}_{n,k}(\psi_R)&:=  \label{Gibbs:eq-scrV-N}
 \sin^2\big(Q_{n,k}+r^{-1} \Proj \psi_R\big) -   \sin^2\big( Q_{n,k} \big)   \\
 &\hspace{0.5ex}-\sin  \big( 2 Q_{n,k} \big) r^{-1} \Proj \psi_R  -  \cos\big( 2 Q_{n,k} \big) (r^{-1} \psi_R)^2. \notag 
 \end{align}
Furthermore, we also define
\begin{equation}
\vec{\nu}^{(N)}_{n,k,R}:= \nu^{(N)}_{n,k,R} \otimes \scrw_R, 
\end{equation}
where $\scrw_R$ is as in Definition \ref{Gaussian:def-white-noise}. 
\end{definition}
 We emphasize that the quadratic term in \eqref{Gibbs:eq-scrV-N} contains $\psi_R$ and not $\Proj \psi_R$. 
 
 \begin{remark}[$\scrg_{n,k,R}$ vs. $\scrg_{0,0,R}$]
 Even for fixed $R\geq R_0$ and $N\geq 1$, it is not entirely obvious that that the Gibbs measure $\nu_{n,k,R}$ from \eqref{Gibbs:eq-Gibbs-def} is well-defined. In order for $\nu_{n,k,R}$ to be well-defined, it is necessary that 
 \begin{equation*}
\exp \Big( \frac{k(k+1)}{2} \int_1^R \dr \, \frac{\cos\big(2Q_{n,k}\big)}{r^2} \psi_R^2(r) \Big) \in L^1\big( \scrg_{n,k,R} \big).
 \end{equation*}
 However, this follows easily from the fact that the covariance operator of $\scrg_{n,k,R}$ is 
 \begin{equation*}
    A_{n,k,R}^{-1} = \Big( - \partial_r^2 + \frac{k(k+1)}{2r^2} \cos\big( 2Q_{n,k} \big) \Big)^{-1}
 \end{equation*}
 and that, for any fixed $R\geq R_0$, $(-\partial_r^2)^{-1}$ is trace-class on $L^2([1,R])$. From similar considerations, it also follows that 
 \begin{equation}\label{Gibbs:eq-switch-Gaussian}
 \begin{aligned}
 \mathrm{d}\nu_{n,k,R}^{(N)} (\psi_R) 
 =& \big( \widetilde{\mathcal{Z}}^{(N)}_{n,k,R}\big)^{-1} \exp \bigg( - \tfrac{k(k+1)}{2} 
\int_1^R \dr \Big( \sin^2( Q_{n,k} + r^{-1} \Proj \psi_R) 
- \sin^2(Q_{n,k}) \\
&\hspace{19ex}- \sin(2Q_{n,k}) r^{-1} \Proj \psi_R \Big) \bigg) \, 
\mathrm{d} \scrg_{0,0,R} \big( \psi_R). 
 \end{aligned}
 \end{equation}
 In other words, the Gibbs measure can also be written with respect to $\scrg_{0,0,R}$ rather than $\scrg_{n,k,R}$. The identity \eqref{Gibbs:eq-switch-Gaussian} is useful when thinking about the invariance of the Gibbs measure for any finite $R\geq R_0$, but will not be useful in the infinite-volume limit $R\rightarrow \infty$. 
 \end{remark}

 We can now state the main proposition of this section, which contains the construction of Gibbs measures on finite and semi-infinite intervals. 
 
 \begin{proposition}[Construction of Gibbs measures]\label{Gibbs:prop-Gibbs}
 Let $n\geq 0$, let $k \geq 1$, let $\alpha:=1/2-\delta$, and let $\kappa:=-1/2-\delta$. Then, we have the following two properties:
 \begin{enumerate}[label=(\roman*)]
     \item\label{Gibbs:item-ultraviolet} (Finite interval) Let $R_0 \leq R <\infty$. As $N\rightarrow \infty$, $\nu^{(N)}_{n,k,R}$ converges in total variation to a unique limit $\nu_{n,k,R}$. Furthermore, it holds that
     \begin{equation*}
    \mathrm{d}\nu_{n,k,R}(\psi_R)
= \mathcal{Z}_{n,k,R}^{-1} \exp \big( - V_{n,k,R}(\psi_R) \big) 
\mathrm{d} \scrg_{n,k,R} \big( \psi_R \big). 
     \end{equation*}

     \item\label{Gibbs:item-infrared} (Semi-infinite interval) There exists a unique probability measure $\nu_{n,k}$ on $\Czero_0([1,\infty))$ which satisfies 
     \begin{equation}\label{Gibbs:eq-infrared-e1}
        \big( \Rest[L][\infty] \big)_\# \nu_{n,k}
    = \operatorname{w-lim}\displaylimits_{R\rightarrow \infty} 
    \big( \Rest[L][R] \big)_\# \nu_{n,k,R}
     \end{equation}
     for all $L\geq 1$. In \eqref{Gibbs:eq-infrared-e1}, the limit refers to the weak limit on  $\Czero_{(0)}([1,L])$. Furthermore, it holds that
     \begin{equation}
        \mathrm{d}\nu_{n,k}(\psi) = \mathcal{Z}_{n,k}^{-1} \exp \Big( - V_{n,k}(\psi) \Big) \mathrm{d}\scrg_{n,k}(\psi). 
     \end{equation}
 \end{enumerate}
 \end{proposition}

In Section \ref{section:Gaussian}, we previously obtained detailed information on the Gaussian measures $\scrg_{n,k,R}$. In order to prove Proposition \ref{Gibbs:prop-Gibbs}, it therefore primarily remains to control the Radon-Nikodym derivative of the Gibbs measures with respect to the Gaussian measures, which is the subject of Section \ref{section:bounds}. The rest of the proof of Proposition \ref{Gibbs:prop-Gibbs}, which is presented in Subsection \ref{section:construction}, relies on soft arguments. 

\subsection{Control of Radon-Nikodym derivative}\label{section:bounds} 

In the first (and main) lemma of this subsection, we prove an exponential moment estimate for the potential energy with respect to the Gaussian measures. 

\begin{lemma}[Uniform exponential bounds]\label{bounds:lem-exponential}
Let $n\geq 0$, let $k\geq 1$, and let $R_0 \leq L \leq R$. Furthermore, let 
\begin{equation}\label{bounds:eq-exponential-q}
0 \leq q < 1 + \frac{1}{4k (k+1)}. 
\end{equation}
Then, we have that
\begin{equation}
\E \Big[ \exp\big( - q V_{n,k,L} \big)  \Big] \lesssim_{q} 1. 
\end{equation}
\end{lemma}

\begin{proof} In the following, we simplify the notation by denoting samples by $\psi$ rather than $\psi_R$. 
It suffices to treat the case 
\begin{equation*}
   1 \leq  q < 1 + \frac{1}{4k (k+1)},
\end{equation*}
since the range $0\leq q <1$ can then be obtained using Hölder's inequality. By using a consequence of the Bou\'{e}-Dupuis formula (Lemma \ref{prelim:lem-boue-dupuis}), it follows that 
\begin{equation}
\begin{aligned}\label{bounds:eq-exponential-p1} 
&- \log \bigg( \E_{\scrg_{n,k,R}} \Big[ \exp\big( - q V_{n,k,L} \big)  \Big] \bigg) \\
\geq& \, \E_{\scrg_{n,k,R}} 
\bigg[ \inf_{\zeta \in \dot{H}_0^1([1,R])} 
\bigg\{ 
q V_{n,k,L}\big( \psi + \zeta \big) 
+ \frac{1}{2} \int_1^R \dr \Big( (\partial_r \zeta)^2 + \frac{k(k+1) \cos\big( 2 Q_{n,k}\big)}{r^2} \zeta^2 \Big)
\bigg\} \bigg]. 
\end{aligned}
\end{equation}
Thus, it suffices to obtain a lower bound on the variational problem in \eqref{bounds:eq-exponential-p1}. In the argument below, the reader should keep the following guiding principle in mind: While Proposition \ref{Gaussian:prop-Gaussian} controls arbitrary moments of the Gaussian process $\psi$, the good term in \eqref{bounds:eq-exponential-p1} only controls the second moment of $\zeta$. As a result, all Taylor expansions should be performed around $\zeta$. \\

We recall that the integral density of  $V_{n,k}(\psi+\zeta)$ is given by a scalar multiple of  
\begin{equation}\label{bounds:eq-exponential-p2}
\begin{aligned}
\scrV_{n,k}(\psi+\zeta) 
&= \sin^2\big(Q_{n,k}+r^{-1} \psi + r^{-1} \zeta \big) - \sin^2\big( Q_{n,k} \big) \\
&- \sin\big( 2 Q_{n,k} \big) r^{-1} (\psi+\zeta)  - \cos\big( 2 Q_{n,k} \big) r^{-2} (\psi+\zeta)^2. 
\end{aligned}
\end{equation}
We now simplify the expressions in \eqref{bounds:eq-exponential-p2}. Using Lemma \ref{prelim:lem-soliton}, the first, second, and third summand in \eqref{bounds:eq-exponential-p2} can be approximated or estimated by 
\begin{align}
\Big| \sin^2\big(Q_{n,k}+r^{-1} \psi + r^{-1} \zeta \big) - 
\sin^2\big(r^{-1} \psi + r^{-1} \zeta \big) \Big| \lesssim \big| Q_{n,k} - n\pi \big| &\lesssim r^{-2}, \label{bounds:eq-exponential-p3} \\
\big| \sin^2\big( Q_{n,k} \big)\big| \lesssim |Q_{n,k}-n\pi| &\lesssim r^{-2}, \label{bounds:eq-exponential-p4} \\ 
\big| \sin\big( 2 Q_{n,k} \big) r^{-1} (\psi+\zeta)  \big| \lesssim  |Q_{n,k}-n\pi| r^{-1} \big( \big| \psi \big| + \big| \zeta \big| \big) &\lesssim r^{-3}  \big( \big| \psi \big| + \big| \zeta \big| \big). \label{bounds:eq-exponential-p5} 
\end{align}
We now combine the $\sin^2(r^{-1}(\psi+\zeta))$-term from \eqref{bounds:eq-exponential-p3} with part of the last summand from \eqref{bounds:eq-exponential-p2}. Using Taylor's theorem, we have that 
\begin{equation}\label{bounds:eq-exponential-p6}
\begin{aligned}
&\Big| \sin^2\big(r^{-1} \psi + r^{-1} \zeta \big) - 2\cos\big( 2 Q_{n,k} \big) r^{-1} \zeta r^{-1} \psi - \cos\big( 2 Q_{n,k} \big) (r^{-1} \psi)^2 \Big|  \\
\lesssim& \, \Big|  \sin^2\big(r^{-1} \psi + r^{-1} \zeta \big) - \sin \big( 2 r^{-1} \zeta \big) r^{-1} \psi - \cos\big( 2 r^{-1} \zeta \big) (r^{-1} \psi)^2 \Big|  \\
+& \, \Big| \sin \big( 2 r^{-1} \zeta \big) r^{-1} \psi - 2 \cos\big( 2 Q_{n,k} \big)  r^{-1} \zeta r^{-1}\psi \Big| \\
+& \, \Big|  \Big( \cos\big( 2 r^{-1} \zeta \big) - \cos\big( 2 Q_{n,k} \big) \Big)  (r^{-1} \psi)^2 \Big| \\
\lesssim& \, \big( r^{-1} \big| \psi \big| \big)^3 + \big| 1 - \cos\big( 2 Q_{n,k}\big) \big| \, r^{-1} |\psi| \,  r^{-1} | \zeta| \\ 
+& \, \Big| \Big( \sin \big( 2 r^{-1} \zeta \big) - 2 r^{-1} \zeta \Big) r^{-1} \psi \Big| + \Big( r^{-1} \big| \zeta \big| + \big| Q_{n,k} - n\pi \big| \Big) \big( r^{-1} |\psi|\big)^2. 
\end{aligned}
\end{equation}
Using Lemma \ref{prelim:lem-soliton}, the elementary estimate $|\sin(x)-x|\lesssim \min(|x|,|x|^3) \lesssim |x|^{3/2}$, and Young's inequality, we obtain for all $\eta\in (0,1)$ that
\begin{equation}
\begin{aligned}\label{bounds:eq-exponential-p7} 
\eqref{bounds:eq-exponential-p6} &\lesssim  \big( r^{-1} \big| \psi \big| \big)^3 + r^{-4} |\psi| |\zeta| + \big( r^{-1} |\zeta| \big)^{3/2} r^{-1} |\psi| + \big( r^{-1} |\zeta| + r^{-2} \big) \big( r^{-1} \psi\big)^2 \\
&\lesssim \eta r^{-2} \zeta^2 + \eta^{-1} r^{-4} \psi^2 + r^{-3} |\psi|^3 + \eta^{-3} r^{-4} \psi^4. 
\end{aligned}
\end{equation}
By combining \eqref{bounds:eq-exponential-p3}-\eqref{bounds:eq-exponential-p7} and using Young's inequality, it follows that 
\begin{equation}\label{bounds:eq-exponential-p8}
\begin{aligned}
&\Big| \scrV_{n,k,R}(\psi+\zeta) - \Big( \sin^2\big( r^{-1} \zeta \big) - \cos\big( 2 Q_{n,k} \big) (r^{-1} \zeta)^2 \Big) \Big| \\
\lesssim& \, r^{-3} |\zeta| + \eta r^{-2} \zeta^2 + \eta^{-3} \Big( r^{-2} + r^{-3} | \psi | + r^{-4} \psi^2 + r^{-3} |\psi|^3 + r^{-4} |\psi|^4 \Big) \\
\lesssim& \, \eta r^{-2} \zeta^2 + \eta^{-3} \Big( r^{-2} + r^{-3} |\psi|^3 + r^{-4} \psi^4 \Big).  
\end{aligned}
\end{equation}
We now let $C_{n,k}\geq 1$ be sufficiently large. 
By inserting \eqref{bounds:eq-exponential-p8} into the objective function in \eqref{bounds:eq-exponential-p1}, it then follows that 
\begin{align}
 &\E_{\scrg_{n,k,R}} \bigg[ \inf_{\zeta \in \dot{H}_0^1([1,R])} 
\bigg\{ 
q V_{n,k,L}\big( \psi + \zeta \big) 
+ \frac{1}{2} \int_1^R \dr \Big( (\partial_r \zeta)^2 + \frac{k(k+1) \cos\big( 2 Q_{n,k}\big)}{r^2} \zeta^2 \Big)
\bigg\} \bigg]  \notag \\ 
 \geq& \,\inf_{\zeta \in \dot{H}_0^1([1,R])} 
\Bigg\{ 
 \frac{1}{2} \int_1^R \dr \, (\partial_r \zeta)^2 + 
  \frac{ q k(k+1)}{2} \int_1^L  \dr  \sin\big( r^{-1} \zeta \big)^2 
  - C_{n,k} \eta  \int_1^R \dr \, r^{-2} \zeta^2 
  \label{bounds:eq-exponential-p9} \\
  &+ \, \frac{k(k+1)}{2}  \int_1^R \dr \,   \frac{\cos\big( 2 Q_{n,k}\big)}{r^2} \zeta^2 - q \frac{k(k+1)}{2}  \int_1^L \dr \,   \frac{\cos\big( 2 Q_{n,k}\big)}{r^2} \zeta^2 \Bigg\}
  \notag  \\
  -&\,  C_{n,k}\eta^{-3} \, \E_{\scrg_{n,k,R}} \bigg[  \int_1^R \dr \Big( r^{-2} + r^{-3} |\psi|^3 + r^{-4} \psi^4 \Big) \bigg] \bigg\}. \label{bounds:eq-exponential-p11}
 \end{align}
 We now treat  \eqref{bounds:eq-exponential-p9} and  \eqref{bounds:eq-exponential-p11} separately. In order to estimate \eqref{bounds:eq-exponential-p9}, we first note that $\sin^2(r^{-1} \zeta)$ is nonnegative, which yields 
 \begin{equation*}
     \frac{ q k(k+1)}{2} \int_1^L  \dr  \sin\big( r^{-1} \zeta \big)^2 \geq 0. 
 \end{equation*}
 Furthermore, since $L\geq R_0$ and $R_0$ is sufficiently large, Lemma \ref{prelim:lem-soliton} implies that $\cos(2Q_{n,k})$ is nonnegative on $[L,R]$. Together with Hardy's inequality (Lemma \ref{prelim:lem-Hardy}), it follows that 
 \begin{align*}
      &\frac{k(k+1)}{2}  \int_1^R \dr \,   \frac{\cos\big( 2 Q_{n,k}\big)}{r^2} \zeta^2 - q \frac{k(k+1)}{2}  \int_1^L \dr \,   \frac{\cos\big( 2 Q_{n,k}\big)}{r^2} \zeta^2 \\
      \geq& \, - (q-1) \frac{k(k+1)}{2} \int_1^R \dr \, \frac{\cos\big( 2 Q_{n,k}\big)}{r^2} \zeta^2  
      \geq\,  - 4 (q-1) \frac{k(k+1)}{2} \int_1^R \dr \, r^{-2} \zeta^2. 
 \end{align*}
 In total, it follows that 
 \begin{equation*}
\eqref{bounds:eq-exponential-p9} \geq \frac{1}{2} \Big( 1- 4 k (k+1) (q-1) - 8 C_{n,k}  \eta \Big) \int_1^R \dr \, (\partial_r \zeta)^2. 
 \end{equation*}
 Due to our assumption on $q$, we can choose $\eta=\eta_{n,k,q}>0$ sufficiently small such that 
 \begin{equation*}
    1- 4 k (k+1) (q-1) - 8 C_{n,k}  \eta >0. 
 \end{equation*}
 Thus, the contribution \eqref{bounds:eq-exponential-p9} is bounded below by zero. In order to complete the proof, it therefore only remains to estimate \eqref{bounds:eq-exponential-p11}. Using Proposition \ref{Gaussian:prop-Gaussian} and our choice of $\eta>0$, it follows for all $\epsilon>0$ that 
 \begin{equation*}
    \eta^{-3} \, \E_{\scrg_{n,k,R}} \bigg[  \int_1^R \dr \Big( r^{-2} + r^{-3} |\psi|^3 + r^{-4} \psi^4 \Big) \bigg] 
    \lesssim_{n,k,q,\epsilon} \int_1^R \dr \Big( r^{-2} + r^{-3/2+\epsilon} + r^{-2+\epsilon} \Big) \lesssim_\epsilon 1,
 \end{equation*}
 which yields the desired lower bound on \eqref{bounds:eq-exponential-p11}.
\end{proof}

\begin{remark}\label{Gibbs:rem-expansion} 
As already discussed in the proof of Lemma \ref{bounds:lem-exponential}, we use a Taylor expansion of the potential energy around the drift term $\zeta$ rather than the Gaussian term $\psi$. This is in sharp contrast to \cite{BG18}, in which the potential energy is expanded around the Gaussian term. 
\end{remark}

While Lemma \ref{bounds:lem-exponential} yields uniform exponential bounds, it does not yield estimates for increments in the interval size $L$ or the frequency-truncation parameter $N$, which are the subject of the next lemma. 

\begin{lemma}[Increments in $L$ and $N$]\label{bounds:lem-increment}
Let $n\geq 0$, let $k\geq 1$, let $R\geq R_0$, let $2\leq L \leq R$, and let $N\geq 1$. We also let $\epsilon>0$ and $p\geq 2$. Then, it holds that 
\begin{align}
\big\| V_{n,k,L} - V_{n,k,L/2 }\big\|_{L^p(\scrg_{n,k,R})} &\lesssim_{\epsilon} p^{3/2} L^{-1/2+\epsilon}, \label{bounds:eq-L} \\ 
\Big\| \big|  V_{n,k,R} -V_{n,k,R}^{(N)}\big| 
\exp\Big(  \big|  V_{n,k,R} -V_{n,k,R}^{(N)}\big|  \Big) \Big\|_{L^{p}(\scrg_{n,k,R})}
&\lesssim_{\epsilon} N^{-1/2+\epsilon}
\exp\big( C_{n,k,\epsilon} R^3 p ),\label{bounds:eq-N}
\end{align}
where $C_{n,k,\epsilon}\geq 1$ is sufficiently large. 
\end{lemma} 

\begin{proof}
 We first prove the estimate for the increment in $L$, i.e., \eqref{bounds:eq-L}. Using Taylor's theorem, the density $\mathscr{V}_{n,k,R}$ from \eqref{prelim:eq-scrV} satisfies
\begin{align*}
\big| \scrV_{n,k}(\psi) \big| 
&\lesssim \Big| 
\sin^2\big(Q_{n,k}+r^{-1} \psi\big) - \sin^2\big( Q_{n,k} \big) - \sin\big( 2 Q_{n,k} \big) r^{-1} \psi  - \cos\big( 2 Q_{n,k} \big) (r^{-1} \psi)^2 \Big| \\
&\lesssim \big| r^{-1} \psi \big|^3. 
\end{align*}
Using the definition of $V_{n,k,L}$ from \eqref{prelim:eq-V}, Hölder's inequality, and Proposition \ref{Gaussian:prop-Gaussian}, this implies
\begin{align*}
&\big\| V_{n,k,L} - V_{n,k,L/2 }\big\|_{L^p(\scrg_{n,k,R})} 
\lesssim \int_{L/2}^L \dr \, \big\| (r^{-1} \psi )^3 \big\|_{L^p(\scrg_{n,k,R})} 
\lesssim \int_{L/2}^L \dr \, \big\| r^{-1} \psi  \big\|_{L^{3p}(\scrg_{n,k,R})}^3 \\
\lesssim& \,  p^{3/2} \int_{L/2}^L \dr \, r^{-3/2+\epsilon} \lesssim L^{-1/2+\epsilon} p^{3/2}. 
\end{align*}
This completes the proof of \eqref{bounds:eq-L} and it remains to prove \eqref{bounds:eq-N}. To this end, we first prove for all $\psi_R \in C^{0,1/2-\epsilon,-1/2-\epsilon}_0([1,R])$ that
\begin{equation}\label{bounds:eq-increment-p1}
    \big|  (V^{(N)}_{n,k,R} - V_{n,k,R})(\psi_R) \big|
    \lesssim_\epsilon R^{3/2} N^{-1/2+\epsilon} \big\| \psi_R \big\|_{C^{0,1/2-\epsilon,-1/2-\epsilon}([1,R])}. 
\end{equation}
Indeed, it follows from Hölder's inequality and Lemma \ref{prelim:lem-projection} that
\begin{align*}
\big|  (V^{(N)}_{n,k,R} - V_{n,k,R})(\psi_R) \big|
&\lesssim \int_1^R \dr \, r^{-1} \big| \Proj \psi_R - \psi_R \big| \\
&\lesssim \big\| \Proj \psi_R - \psi_R\big\|_{L^2([1,R])}\\
&\lesssim_\epsilon R^{1-\epsilon} N^{-1/2+\epsilon} \big\| \psi_R \big\|_{C^{0,1/2-\epsilon,0}([1,R])} \\
&\lesssim_\epsilon R^{3/2} N^{-1/2+\epsilon} \big\| \psi_R \big\|_{C^{0,1/2-\epsilon,-1/2-\epsilon}([1,R])}. 
\end{align*}
Using \eqref{bounds:eq-increment-p1} and that $C_{n,k,\epsilon}\geq 1$ is sufficiently large, it follows that
\begin{align*}
 &\Big\| \big|  V_{n,k,R} -V_{n,k,R}^{(N)}\big| 
\exp\Big(  \big|  V_{n,k,R} -V_{n,k,R}^{(N)}\big|  \Big) \Big\|_{L^{p}(\scrg_{n,k,R})} \\
\lesssim_\epsilon & \, N^{-1/2+\epsilon} 
\Big\| R^{3/2} \big\| \psi_R \big\|_{C^{0,1/2-\epsilon,-1/2-\epsilon}([1,R])}
 \exp\Big( \tfrac{C_{n,k,\epsilon}}{20} N^{-1/2+\epsilon} R^{3/2} \big\| \psi_R \big\|_{C^{0,1/2-\epsilon,-1/2-\epsilon}([1,R])}  \Big) \Big\|_{L^{p}(\scrg_{n,k,R})} \\
 \lesssim_\epsilon& \,  N^{-1/2+\epsilon}  
 \Big\| \exp\Big( \tfrac{C_{n,k,\epsilon}}{10} R^{3/2} \big\| \psi_R \big\|_{C^{0,1/2-\epsilon,-1/2-\epsilon}([1,R])}  \Big) \Big\|_{L^{p}(\scrg_{n,k,R})}
\end{align*}
Thus, the desired estimate follows from Proposition \ref{Gaussian:prop-Gaussian} (and exponential moment estimates for sub-Gaussian random variables). 
\end{proof}

At the end of this subsection, we record the following corollary of Lemma \ref{bounds:lem-exponential} and Lemma \ref{bounds:lem-increment}, which is used to control the normalization constants.

\begin{corollary}\label{Gibbs:cor-normalization}
Let $n\geq 0$ and let $k\geq 1$. Then, it holds that 
\begin{equation}\label{Gibbs:eq-normalization}
\E_{\scrg_{n,k,R}} \Big[ \exp\big( - V_{n,k,L} \big) \Big] \sim 1 
\end{equation}
uniformly for all $R_0 \leq L \leq R$. 
\end{corollary}

\begin{proof}
The upper bound in \eqref{Gibbs:eq-normalization} follows directly from Lemma \ref{bounds:lem-exponential}. Using Jensen's inequality and Lemma \ref{bounds:lem-increment}, we also have that 
\begin{equation*}
\E_{\scrg_{n,k,R}} \Big[ \exp\big( - V_{n,k,L} \big) \Big] \geq \exp \Big( - \E_{\scrg_{n,k,R}} \big[ V_{n,k,L} \big] \Big) \gtrsim 1, 
\end{equation*}
which yields the lower bound in \eqref{Gibbs:eq-normalization}. 
\end{proof}

\subsection{Proof of Proposition \ref{Gibbs:prop-Gibbs}}\label{section:construction}

Equipped with the estimates from Subsection \ref{section:bounds}, we now present the proof of Proposition \ref{Gibbs:prop-Gibbs}. 

\begin{proof}[Proof of Proposition \ref{Gibbs:prop-Gibbs}:]
We first construct the Gibbs measures for finite intervals, i.e., we first  prove \ref{Gibbs:item-ultraviolet}. Due to Corollary \ref{Gibbs:cor-normalization}, it suffices to prove that 
\begin{equation}\label{Gibbs:eq-Gibbs-p1}
\lim_{N\rightarrow \infty} \Big\| \exp\big( - V_{n,k,R} \big) - \exp\big( -V_{n,k,R}^{(N)}\big) \Big\|_{L^1(\scrg_{n,k,R})} =0. 
\end{equation}
To this end, we let $q=q_k \in (1,\infty)$ satisfy \eqref{bounds:eq-exponential-q} and let $q^\prime$ be its Hölder-conjugate. Using the elementary estimate 
\begin{equation*}
|\exp(-x)-\exp(-y)| \lesssim |x-y| \exp\big( |x-y|\big) \exp\big( -x\big) \qquad \forall x,y \in \R
\end{equation*}
and Hölder's inequality, it follows that
\begin{align*}
&\Big\| \exp\big( - V_{n,k,R} \big) - \exp\big( -V_{n,k,R}^{(N)}\big) \Big\|_{L^1(\scrg_{n,k,R})} \\ 
\lesssim \, &\Big\| \big|  V_{n,k,R} -V_{n,k,R}^{(N)}\big| 
\exp\Big(  \big|  V_{n,k,R} -V_{n,k,R}^{(N)}\big|  \Big) 
\exp\Big( - V_{n,k,R} \Big)  \Big\|_{L^1(\scrg_{n,k,R})} \\
\lesssim\, & \Big\| \big|  V_{n,k,R} -V_{n,k,R}^{(N)}\big| 
\exp\Big(  \big|  V_{n,k,R} -V_{n,k,R}^{(N)}\big|  \Big) \Big\|_{L^{q^\prime}(\scrg_{n,k,R})}
\, \Big\|\exp\Big( - V_{n,k,R} \Big) \Big\|_{L^q(\scrg_{n,k,R})}
\end{align*}
By using Lemma \ref{bounds:lem-exponential} and Lemma \ref{bounds:lem-increment}, it follows for all $\epsilon>0$ that
\begin{align*}
\Big\| \big|  V_{n,k,R} -V_{n,k,R}^{(N)}\big| 
\exp\Big(  \big|  V_{n,k,R} -V_{n,k,R}^{(N)}\big|  \Big) \Big\|_{L^{q^\prime}(\scrg_{n,k,R})}
\, \Big\|\exp\Big( - V_{n,k,R} \Big) \Big\|_{L^q(\scrg_{n,k,R})}
\lesssim_{n,k,R,\epsilon} \, & N^{-1/2+\epsilon}. 
\end{align*}
This completes the proof of \eqref{Gibbs:eq-Gibbs-p1}. \\ 

We now construct the Gibbs measure on the semi-infinite interval, i.e., we now prove \ref{Gibbs:item-infrared}.  Using Lemma \ref{bounds:lem-exponential}, we can define
\begin{equation}\label{Gibbs:eq-infrared-p1}
\mathrm{d}\nu_{n,k}(\psi):= \mathcal{Z}_{n,k}^{-1} \exp\big( - V_{n,k} (\psi) \big) \mathrm{d}\scrg_{n,k}(\psi). 
\end{equation}
In order to prove \eqref{Gibbs:eq-infrared-e1}, we introduce auxiliary probability  measures. To be more precise, we let $R_0 \leq R^\prime \leq R$ and define a probability measure $\nu_{n,k,R,R^\prime}$ on $\Czero_0([1,R])$ 
by  
\begin{equation}\label{Gibbs:eq-infrared-p2} 
    \mathrm{d}\nu_{n,k,R,R^\prime}(\psi_R):= \mathcal{Z}_{n,k,R,R^\prime}^{-1}  
     \exp\big( - V_{n,k,R^\prime} (\psi_R) \big) \mathrm{d}\scrg_{n,k,R}(\psi_R). 
\end{equation}
We note that the difference between $\nu_{n,k,R,R^\prime}$ and $\nu_{n,k,R}$ is that the potential energy is only integrated over $[1,R^\prime]$ rather than $[1,R]$. We now claim for all $1\leq R^\prime \leq R \leq \infty$ and all $\epsilon>0$  that
\begin{equation}\label{Gibbs:eq-infrared-p3}
\big\| \nu_{n,k,R,R^\prime} - \nu_{n,k,R} \big\|_{\textup{TV}} \lesssim_\epsilon (R^\prime)^{-1/2+\epsilon}. 
\end{equation}
In order to prove \eqref{Gibbs:eq-infrared-p3}, we first recall from Corollary \ref{Gibbs:cor-normalization} that $\mathcal{Z}_{n,k,R,R^\prime}\sim 1$ (uniformly in $R$ and $R^\prime$). As a result, it holds that 
\begin{equation}\label{Gibbs:eq-infrared-p4} 
\big\| \nu_{n,k,R,R^\prime} - \nu_{n,k,R} \big\|_{\textup{TV}} 
\lesssim \Big\| \exp\big( - V_{n,k,R}\big) - \exp\big( - V_{n,k,R^\prime}\big)
\Big\|_{L^1(\scrg_{n,k,R})}. 
\end{equation}
We now choose any $q=q_k\in (1,\infty)$ satisfying \eqref{bounds:eq-exponential-q} and let $q^\prime$ be its Hölder-conjugate.  Using the elementary estimate 
\begin{equation*}
    |\exp(-x)-\exp(-y)|\lesssim |x-y| \big(\exp(-x)+\exp(-y)\big) \qquad \forall x,y\in \R 
\end{equation*}
and Hölder's inequality, it follows that 
\begin{align*}
&\,\Big\| \exp\big( - V_{n,k,R}\big) - \exp\big( - V_{n,k,R^\prime}\big)
\Big\|_{L^1(\scrg_{n,k,R})} \\ 
\lesssim&\,  \Big\| \big|V_{n,k,R}- V_{n,k,R^\prime}\big| \Big( \exp\big( - V_{n,k,R}\big) + \exp\big( - V_{n,k,R^\prime}\big) \Big) 
\Big\|_{L^1(\scrg_{n,k,R})} \\ 
\lesssim&\,  \Big\| V_{n,k,R}- V_{n,k,R^\prime} \Big\|_{L^{q^\prime}(\scrg_{n,k,R})}
\cdot \Big\|  \exp\big( - V_{n,k,R}\big) + \exp\big( - V_{n,k,R^\prime}\big)  \Big\|_{L^q(\scrg_{n,k,R})}. 
\end{align*}
After using Lemma \ref{bounds:lem-exponential} and Lemma \ref{bounds:lem-increment}, this completes the proof of the claim \eqref{Gibbs:eq-infrared-p3}. Due to \eqref{Gibbs:eq-infrared-p3}, it now only remains to prove that 
\begin{equation}\label{Gibbs:eq-infrared-p5}
        \big( \Rest[L][\infty] \big)_\# \nu_{n,k,\infty,R^\prime}
    = \operatorname{w-lim}\displaylimits_{R\rightarrow \infty} 
    \big( \Rest[L][R] \big)_\# \nu_{n,k,R,R^\prime}, 
\end{equation}
where the limit refers to the weak limit on $\Czero_{(0)}([1,L])$. In order to prove \eqref{Gibbs:eq-infrared-p5}, it suffices\footnote{To see this, one only has to realize that \eqref{Gibbs:eq-infrared-p6} with $f=1$ implies the convergence of the normalization constants $\mathcal{Z}_{n,k,R,R^\prime}$ as $R\rightarrow \infty$. Once the convergence of the normalization constants is established, the equivalence of \eqref{Gibbs:eq-infrared-p5} and \eqref{Gibbs:eq-infrared-p6} is clear.} to show that
\begin{equation}\label{Gibbs:eq-infrared-p6}
\begin{aligned}
&\lim_{R\rightarrow \infty} \int f\big( \Rest[L][R] \psi_R \big) \exp\Big( - V_{n,k,R^\prime}(\psi_R) \Big) \mathrm{d}\scrg_{n,k,R}(\psi_R) \\ 
=& \,   \int f\big( \Rest[L][\infty] \psi \big) \exp\Big( - V_{n,k,R^\prime}(\psi) \Big) \mathrm{d}\scrg_{n,k}(\psi)
\end{aligned}
\end{equation}
for all $L,R^\prime \geq 1$ and bounded and Lipschitz continuous $f\colon \Czero_{(0)}([1,L])\rightarrow \R$. Since $\Rest[L][R]=\Rest[L][R^\prime]\circ \Rest[R^\prime][R]$
and $V_{n,k,R^\prime}(\psi_R)=V_{n,k,R^\prime}(\Rest[R^\prime][R]\psi_R)$, the left-hand side of \eqref{Gibbs:eq-infrared-p6} can be rewritten as  
\begin{align*}
&\int f\big( \Rest[L][R] \psi_R \big) \exp\Big( - V_{n,k,R^\prime}(\psi_R) \Big) \mathrm{d}\scrg_{n,k,R}(\psi_R) \\ 
=& \int \big( f \circ \Rest[L][R^\prime] \big) \big( \Rest[R^\prime][R] \psi_R \big) \exp\Big( - V_{n,k,R^\prime}(\Rest[R^\prime][R] \psi_R) \Big) \mathrm{d}\scrg_{n,k,R}(\psi_R) \\
=& \int \Big( \big( f \circ \Rest[L][R^\prime] \big) \cdot \exp\big( - V_{n,k,R^\prime} \big)  \Big)\big( \psi_{R^\prime} \big) \, \mathrm{d}\Big( \big( \Rest[R^\prime][R]\big)_{\#} \scrg_{n,k,R}\Big)(\psi_{R^\prime}). 
\end{align*}
Similarly, the right-hand side of \eqref{Gibbs:eq-infrared-p6} can be written as 
\begin{align*}
&\int f\big( \Rest[L][\infty] \psi \big) \exp\Big( - V_{n,k,R^\prime}(\psi) \Big) \mathrm{d}\scrg_{n,k}(\psi)\\
=&\, \int \Big( \big( f \circ \Rest[L][R^\prime] \big) \cdot \exp\big( - V_{n,k,R^\prime} \big)  \Big)\big( \psi_{R^\prime} \big) \, \mathrm{d}\Big( \big( \Rest[R^\prime][\infty]\big)_{\#} \scrg_{n,k}\Big)(\psi_{R^\prime}).
\end{align*}
As a result, \eqref{Gibbs:eq-infrared-p6} is equivalent to 
\begin{equation}\label{Gibbs:eq-infrared-p7}
\begin{aligned}
&\lim_{R\rightarrow \infty}\int \Big( \big( f \circ \Rest[L][R^\prime] \big) \cdot \exp\big( - V_{n,k,R^\prime} \big)  \Big)\big( \psi_{R^\prime} \big) \, \mathrm{d}\Big( \big( \Rest[R^\prime][R]\big)_{\#} \scrg_{n,k,R}\Big)(\psi_{R^\prime}) \\ 
=& \,   \int \Big( \big( f \circ \Rest[L][R^\prime] \big) \cdot \exp\big( - V_{n,k,R^\prime} \big)  \Big)\big( \psi_{R^\prime} \big) \, \mathrm{d}\Big( \big( \Rest[R^\prime][\infty]\big)_{\#} \scrg_{n,k}\Big)(\psi_{R^\prime}). 
\end{aligned}
\end{equation}
Since the identity \eqref{Gibbs:eq-infrared-p7} follows directly from the weak convergence of the Gaussian measures (Lemma \ref{Gaussian:lem-infinite}), the Lipschitz continuity of $V_{n,k,R^\prime}$ (for any fixed $R^\prime$), and the exponential moment estimates (Lemma \ref{bounds:lem-exponential}), this completes the proof.  
\end{proof}

\section{Dynamics}\label{section:dynamics}

In this section, we address the dynamical aspects of Theorem \ref{intro:thm-main}. In Subsection \ref{section:gwp}, we prove the global well-posedness of the equivariant wave maps equation in weighted Hölder spaces. We emphasize that, as previously discussed in the introduction, the well-posedness theory does not rely on any probabilistic properties of the initial data. In Subsection \ref{section:invariance-finite} and Subsection \ref{section:invariance-semi-infinite}, we prove the invariance of the Gibbs measure for the finite intervals $[1,R]$ and the semi-infinite interval $[1,\infty)$, respectively. The main ingredients are the finite-dimensional approximation from Subsection \ref{section:prelim-finite-dimensional} and finite speed of propagation. 

\subsection{Global well-posedness}\label{section:gwp}

In this subsection, we prove all necessary ingredients for the global well-posedness of the equivariant wave maps equation \eqref{intro:eq-phi}. In the unknown $\psi$ from \eqref{prelim:eq-change-of-variables}, the initial-boundary value problems on the semi-infinite and finite intervals are given by  
\begin{equation}\label{dynamics:eq-psi-infty}
\begin{cases}
\begin{alignedat}{3}
\partial_t^2 \psi - \partial_r^2 \psi &= - r^{-1} \Nl\big( r^{-1} \psi \big) \hspace{5ex}
(&&t,r) \in \R \times (1,\infty), \\ 
\psi(t,1) &=0 &&t \in \mathbb{R}, \\ 
\lim_{r\rightarrow \infty} r^{-1} \psi(t,r) &=0  &&t \in \mathbb{R}, \\ 
\big( \psi, \partial_t \psi \big)(0,r) &= (\psi_0,\psi_1)(r) &&r \in (1,\infty),
\end{alignedat}
\end{cases}
\end{equation}
and 
\begin{equation}\label{dynamics:eq-psi-R}
\begin{cases}
\begin{alignedat}{3}
\partial_t^2 \psi_R - \partial_r^2 \psi_R &= - r^{-1} \Nl\big( r^{-1} \psi_R \big) \hspace{5ex} (&&t,r) \in \R \times (1,R), \\ 
\psi_R(t,1) &=0 && t \in \mathbb{R}, \\ 
\psi_R(t,R) &=0 && t \in \mathbb{R}, \\ 
\big( \psi_R, \partial_t \psi_R \big)(0,r) &= (\psi_{R,0},\psi_{R,1})(r) && r \in (1,R).
\end{alignedat}
\end{cases}
\end{equation}
Here, the nonlinearity $\Nl$ is as in \eqref{prelim:eq-Nl}.

\begin{proposition}[Global well-posedness of \eqref{dynamics:eq-psi-R}]\label{dynamics:prop-gwp}
Let $1\leq R < \infty$, let $\alpha \in [0,1)$, and let $-1<\kappa \leq 0$. Then, \eqref{dynamics:eq-psi-R} is globally well-posed in 
$( \Czero_0 \times \Cmone )( [1,R])$
and the unique global solution $\psi$ satisfies
\begin{equation}\label{dynamics:eq-gwp-a-priori}
\big\| (\psi_R,\partial_t \psi_R) \big\|_{(\Czero_0 \times \Cmone)( [1,R])}  \lesssim \langle t \rangle^{|\kappa|} 
\big\| (\psi_{R,0}, \psi_{R,1})\big\|_{( \Czero_0 \times \Cmone )( [1,R])} + \langle t\rangle^2
\end{equation}
for all $t\in \R$. After the obvious modifications, a similar statement also holds in the semi-infinite case $R=\infty$. 
\end{proposition}

The condition $\kappa>-1$ is only imposed in order to satisfy the growth condition as $r\rightarrow \infty$ in \eqref{dynamics:eq-psi-infty}. Due to Proposition \ref{dynamics:prop-gwp}, we can introduce the global flows
\begin{equation}\label{dynamics:eq-flow-psi}
\vec{\Psi}= (\Psi_0,\Psi_1) \colon 
\R \times \Czero_0([1,\infty)) \times \Cmone([1,\infty))
\rightarrow  \Czero_0([1,\infty)) \times \Cmone([1,\infty))
\end{equation}
and 
\begin{equation}\label{dynamics:eq-flow-psi-R}
\vec{\Psi}_R= (\Psi_{R,0},\Psi_{R,1}) \colon 
\R \times \Czero_0([1,R]) \times \Cmone([1,R])
\rightarrow \Czero_0([1,R]) \times \Cmone([1,R])
\end{equation}
corresponding to \eqref{dynamics:eq-psi-infty} and \eqref{dynamics:eq-psi-R}, respectively. 
Before we turn to the proof of Proposition \ref{dynamics:prop-gwp}, we record the following estimates for homogeneous and inhomogeneous linear waves. In addition to the proof of Proposition \ref{dynamics:prop-gwp}, these estimates will also be used in Subsection \ref{section:invariance-finite} below.

\begin{lemma}[Linear estimates]\label{dynamics:lem-linear}
Let $1\leq R <\infty$, let $\alpha \in [0,1)$, and let $\kappa \leq 0$. Then, we have the following estimates: 
\begin{enumerate}[label=(\roman*)]
\item\label{dynamics:item-linear-1} (Linear wave estimate) For all $\psi_{R,0} \in \Czero([1,R])$ and $t\in \R$, it holds that 
\begin{equation*}
\Big\| \psi_{R,0} \big( \ext_R(r\pm t)\big) \Big\|_{\Czero([1,R])}
+ \Big\| \partial_t \Big( \psi_{R,0} \big( \ext_R(r\pm t) \big) \Big) \Big\|_{\Cmone([1,R])} 
\lesssim \langle t \rangle^{|\kappa|} \big\| \psi_{R,0} \big\|_{\Czero([1,R])}.
\end{equation*}
\item\label{dynamics:item-linear-2} ($L^\infty$-based Duhamel estimate) For all $T\geq 0$, $t\in [-T,T]$, and $F \in L^\infty([-T,T]\times [1,R])$, it holds that 
\begin{align*}
\Big\| \Duh \big[ F \big](t) \Big\|_{\Czero([1,R])}
+ \Big\| \partial_t \Duh \big[ F \big](t) \Big\|_{\Cmone([1,R])} 
\lesssim |t| \langle T \rangle \big\| F \big\|_{L^\infty([-T,T]\times [1,R])}. 
\end{align*}
\item\label{dynamics:item-linear-3} ($L^2$-based Duhamel estimate)  
Assume that $\alpha \leq 1/2$. Then, it holds for  all $T\geq 0$, $t\in [-T,T]$, $F \in L^1_s L^2_\rho([-T,T]\times [1,R])$, and $\alpha \leq 1/2$ that
\begin{align*}
\Big\| \Duh \big[ F \big](t) \Big\|_{\Czero([1,R])}
+ \Big\| \partial_t \Duh \big[ F \big](t) \Big\|_{\Cmone([1,R])} 
\lesssim& \,  \langle t \rangle \big\| F \big\|_{L^1_s L^2_\rho([-T,T]\times [1,R])}. 
\end{align*}
\end{enumerate}
After the obvious modifications, similar estimates also hold in the semi-infinite case $R=\infty$. 
\end{lemma}

\begin{proof}
We separate the proofs of \ref{dynamics:item-linear-1}, \ref{dynamics:item-linear-2}, and \ref{dynamics:item-linear-3}. \\

\emph{Proof of \ref{dynamics:item-linear-1}:} 
The estimate of the $\Czero$-norm follows directly from the Lipschitz continuity of $\ext_R$ and $\ext_R(r)=r$ for all $r\in [1,R]$ (as stated in Lemma \ref{prelim:lem-extension}). Since 
\begin{equation*}
\partial_t \Big( \psi_{R,0} \big( \ext_R ( r \pm t ) \big) \Big)
= \pm \partial_r \Big( \psi_{R,0} \big( \ext_R ( r \pm t ) \big) \Big),
\end{equation*}
the $\Cmone$-estimate for the time-derivative follows directly from the $\Czero$-estimate.\\ 

\emph{Proof of \ref{dynamics:item-linear-2}:}
We first prove the $\Czero$-estimate. To this end, we first bound
\begin{align*}
\Big| \Duh \big[ F \big](t,r) \Big| 
&= \Big| \int_0^t \ds \int_{r-(t-s)}^{r+(t-s)} \drho \, 
(\Ext_R F)(s,\rho) \Big| \\ 
&\leq \Big| \int_0^t \ds \int_{r-(t-s)}^{r+(t-s)} \drho \, 1 \Big| \times  \big\| (\Ext_R F)(s,\rho) \big\|_{L^\infty([-T,T] \times \R)} \\ 
&\leq t^2 \big\| F \big\|_{L^\infty([-T,T]\times [1,R])}. 
\end{align*}
For any $r,r^\prime \in [1,R]$, we further estimate
\begin{align*}
    &\Big| \Duh \big[ F \big](t,r) - \Duh \big[ F \big](t,r^\prime) \Big| \\
    =& 
\Big| \int_0^t \ds \int_{r-(t-s)}^{r+(t-s)} \drho \, 
(\Ext_R F)(s,\rho) - 
 \int_0^t \ds \int_{r^\prime-(t-s)}^{r^\prime+(t-s)} \drho \, 
(\Ext_R F)(s,\rho) \Big| \\ 
\leq& \Big|  \int_0^t \ds \int_{r^\prime+(t-s)}^{r+(t-s)} \drho \, 
(\Ext_R F)(s,\rho) \Big| 
+ \Big|  \int_0^t \ds \int_{r^\prime-(t-s)}^{r-(t-s)} \drho \, 
(\Ext_R F)(s,\rho) \Big|  \\ 
\leq & 4 |t| |r-r^\prime|  \big\| (\Ext_R F)(s,\rho) \big\|_{L^\infty([-T,T] \times \R)} \\
\leq & 4 |t| |r-r^\prime| \big\| F \big\|_{L^\infty([-T,T]\times [1,R])}. 
\end{align*}
Since $\varkappa \leq 0$, this completes the proof of the $\Czero$-estimate. In order to prove the $\Cmone$-estimate for the time-derivative, we first observe that 
\begin{equation}\label{prelim:eq-linear-p1}
\begin{aligned}
&\partial_t \int_0^t \ds \int_{r-(t-s)}^{r+(t-s)} \drho \, 
(\Ext_R F)(s,\rho)  \\
=& \partial_r \int_0^t \ds \int_{r}^{r+(t-s)} \drho \, 
(\Ext_R F)(s,\rho) 
- \partial_r \int_0^t \ds \int_{r}^{r-(t-s)} \drho \, 
(\Ext_R F)(s,\rho) . 
\end{aligned}
\end{equation}
Due to the definition of the $\Cmone$-norm, which contains an integral, the estimate of the $\Cmone$-norm of the time-derivative can be deduced similarly as the $\Czero$-estimate above. \\

\emph{Proof of \ref{dynamics:item-linear-3}:} 
Due to the identity \eqref{prelim:eq-linear-p1}, it suffices to prove the $\Czero$-estimate. To this end, we first prove a pointwise estimate. It holds that 
\begin{equation}\label{prelim:eq-linear-p2}
\begin{aligned}
\Big| \Duh \big[ F \big](t,r) \Big| 
&= \Big| \int_0^t \ds \int_{r-(t-s)}^{r+(t-s)} \drho \, 
(\Ext_R F)(s,\rho) \Big| \\ 
&\leq  \int_0^t \ds \,  \bigg( \int_{r-(t-s)}^{r+(t-s)} \drho \, 1 \bigg)^{1/2}   \big\| (\Ext_R F)(s,\rho) \big\|_{L^2_\rho([r-(t-s),r+(t-s)])}  \\ 
&\lesssim \langle  t \rangle^{1/2} \int_0^t \ds \,  \big\| (\Ext_R F)(s,\rho) \big\|_{L^2_\rho([r-(t-s),r+(t-s)])}. 
\end{aligned}
\end{equation}
Furthermore, from the definition of the extension operator $\Ext_R$, it follows that 
\begin{equation}\label{prelim:eq-linear-p3} 
  \big\| (\Ext_R F)(s,\rho) \big\|_{L^2_\rho([r-(t-s),r+(t-s)])} 
  \lesssim \Big( 1 + \frac{|t-s|}{R} \Big)^{1/2}
  \big\| F \big\|_{L^2_\rho([1,R])} \lesssim \langle t \rangle^{1/2} \big\| F \big\|_{L^2_\rho([1,R])}. 
\end{equation}
By combining \eqref{prelim:eq-linear-p2} and \eqref{prelim:eq-linear-p3}, we obtain that 
\begin{equation}\label{prelim:eq-linear-p4}
\big| \Duh \big[ F \big](t,r)\big| \lesssim \langle t \rangle \big\| F \big\|_{L^1_s L^2_\rho([0,t]\times [1,R])}. 
\end{equation}
Similarly, it holds for all $1\leq r^\prime \leq r \leq R$ that 
\begin{align*}
        &\Big| \Duh \big[ F \big](t,r) - \Duh \big[ F \big](t,r^\prime) \Big| \\
    =& 
\Big| \int_0^t \ds \int_{r-(t-s)}^{r+(t-s)} \drho \, 
(\Ext_R F)(s,\rho) - 
 \int_0^t \ds \int_{r^\prime-(t-s)}^{r^\prime+(t-s)} \drho \, 
(\Ext_R F)(s,\rho) \Big| \\ 
\leq& \Big|  \int_0^t \ds \int_{r^\prime+(t-s)}^{r+(t-s)} \drho \, 
(\Ext_R F)(s,\rho) \Big| 
+ \Big|  \int_0^t \ds \int_{r^\prime-(t-s)}^{r-(t-s)} \drho \, 
(\Ext_R F)(s,\rho) \Big|  \\ 
\lesssim& |r-r^\prime|^{1/2} \int_0^t \ds \Big( \Big\| (\Ext_R F)(s,\rho)  \Big\|_{L^2_\rho([r^\prime+(t-s),r+(t-s)])} 
+ \Big\| (\Ext_R F)(s,\rho)  \Big\|_{L^2_\rho([r^\prime-(t-s),r-(t-s)])}\Big). 
\end{align*}
Since $r,r^\prime \in [1,R]$, it holds that 
\begin{align*}
\Big\| (\Ext_R F)(s,\rho)  \Big\|_{L^2_\rho([r^\prime+(t-s),r+(t-s)])} 
+ \Big\| (\Ext_R F)(s,\rho)  \Big\|_{L^2_\rho([r^\prime-(t-s),r-(t-s)])}
\lesssim \big\| F \big\|_{L^2_\rho([1,R])}. 
\end{align*}
Thus, it follows that
\begin{equation}\label{prelim:eq-linear-p5}
\Big| \Duh \big[ F \big](t,r) - \Duh \big[ F \big](t,r^\prime) \Big| \lesssim |r-r^\prime|^{1/2} \big\| F \big\|_{L^1_s L^2_\rho([0,t]\times [1,R])}. 
\end{equation}
Since $\alpha \leq 1/2$ and $\kappa \leq 0$, \eqref{prelim:eq-linear-p4} and \eqref{prelim:eq-linear-p5} imply the desired estimate.
\end{proof}

Equipped with Lemma \ref{dynamics:lem-linear}, we are now ready to prove Proposition \ref{dynamics:prop-gwp}. 

\begin{proof}[Proof of Proposition \ref{dynamics:prop-gwp}]
We only prove the a-priori estimate \eqref{dynamics:eq-gwp-a-priori}, since the remaining claims follow from a standard contraction mapping argument. By time-reversal symmetry, it suffices to prove \eqref{dynamics:eq-gwp-a-priori} for $t\geq 0$. To this end, let $\psi_R$ be a global solution of \eqref{dynamics:eq-psi-R}. Using Lemma \ref{prelim:lem-dAlembert}, it follows that 
\begin{align*}
\psi_R(t,r) &= \frac{\big( \Ext_R \psi_{R,0} \big)(r+t)+ \big( \Ext_R \psi_{R,0} \big)(r-t)}{2} 
+\frac{1}{2} \int_{r-t}^{r+t} \drho \big( \Ext \psi_{R,1} \big)(\rho) \\ 
&+ \Duh \Big[ r^{-1} \Nl\big( r^{-1} \psi_R \big)\Big](t,r). 
\end{align*}
Using Lemma \ref{dynamics:lem-linear},  it follows that 
\begin{align*}
&\big\| (\psi_R,\partial_t \psi_R) \big\|_{(\Czero_0 \times \Cmone)( [1,R])}  \\
\lesssim&\,  \langle t \rangle^{|\kappa|} 
\big\| (\psi_{R,0}, \psi_{R,1})\big\|_{( \Czero_0 \times \Cmone )( [1,R])} + \langle t \rangle^2 \Big\| \rho^{-1} \Nl\big( \rho^{-1} \psi_R \big)\Big\|_{L_s^\infty L_\rho^\infty([0,t]\times [1,R])}. 
\end{align*}
Using the crude estimates $\rho^{-1}\leq 1$ and $|\Nl|\leq 1$, this yields the desired estimate. 
\end{proof}

\subsection{Invariance for finite intervals}\label{section:invariance-finite}

In the previous subsection, we established the global well-posedness of \eqref{dynamics:eq-psi-infty} and \eqref{dynamics:eq-psi-R}. In contrast to the proof of global well-posedness, however, our proof of invariance of the Gibbs measure treats the finite and semi-infinite interval separately. In this subsection, we treat finite intervals.

\begin{proposition}[Invariance for finite-intervals]\label{dynamics:prop-invariance-finite}
Let $n\geq 0$, let $k\geq 1$, and let $R_0 \leq R <\infty$. Then, the Gibbs measure $\vec{\nu}_{n,k,R}$ is invariant under the dynamics of \eqref{dynamics:eq-psi-R}. 
\end{proposition}

In Subsection \ref{section:prelim-finite-dimensional}, we introduced a finite-dimensional approximation of $\psi_R$, which is given by 
\begin{equation}\label{dynamics:eq-psi-R-N}
\begin{cases}
\begin{alignedat}{3}
\big( \partial_t^2 - \partial_r^2 \big) \psi_{R}^{(N)} &= - \Proj \Big( r^{-1} \Nl\big( r^{-1} \Proj \psi_{R}^{(N)} \big) \Big) \hspace{5ex} (&&t,r) \in \R \times (1,R), \\ 
\psi_{R}^{(N)}(t,1) &=0  &&t \in \mathbb{R}, \\ 
\psi_{R}^{(N)}(t,R) &=0 &&t \in \mathbb{R}, \\ 
\big( \psi_{R}^{(N)}, \partial_t \psi_{R}^{(N)} \big)(0,r) &= (\psi_{R,0},\psi_{R,1})(r) &&r \in (1,R). 
\end{alignedat}
\end{cases}
\end{equation}
The main ingredient in the proof of Proposition \ref{dynamics:prop-invariance-finite} is the following approximation lemma.

\begin{lemma}[Approximation lemma]\label{dynamics:lem-approximation} 
Let $R\geq 1$ be fixed, let $1\leq N<\infty$, let $0<\alpha<1/2$, and let $\kappa \leq 0$. Then, \eqref{dynamics:eq-psi-R-N} is globally well-posed in $(\Czero_0 \times \Cmone)([1,R])$. Furthermore, let $\psi_R$ and $\psi^{(N)}_R$ be the unique global solutions of \eqref{dynamics:eq-psi-R} and \eqref{dynamics:eq-psi-R-N}, respectively. For all $T\geq 0$ and $t\in [-T,T]$, it then holds that
\begin{equation}\label{dynamics:eq-approximation}
\begin{aligned}
&\Big\| (\psi_R, \partial_t \psi_R)(t) - (\psi_{R}^{(N)}, \partial_t \psi_{R}^{(N)})(t) \Big\|_{(\Czero_0 \times \Cmone)([1,R])} \\ 
\leq& \,  C(R,T)   \Big( 1 + \big\| (\psi_{R,0},\psi_{R,1}) \big\|_{(\Czero_0 \times \Cmone)([1,R])} \Big) N^{-\alpha}. 
\end{aligned}
\end{equation}
\end{lemma}

Similar as in \eqref{dynamics:eq-flow-psi} and \eqref{dynamics:eq-flow-psi-R} above, we denote the corresponding global flow by 
\begin{equation}\label{dynamics:eq-flow-psi-R-N}
\vec{\Psi}_R^{(N)}= (\Psi_{R,0}^{(N)},\Psi_{R,1}^{(N)}) \colon 
\R \times \Czero_0([1,R]) \times \Cmone([1,R])
\rightarrow \Czero_0([1,R]) \times \Cmone([1,R]). 
\end{equation}

\begin{proof}
Since $R\geq 1$ is fixed, we simplify the notation by writing $\psi$ and $\psi^{(N)}$ instead of $\psi_R$ and $\psi^{(N)}_R$, respectively. 
Due to the soft estimate \eqref{prelim:eq-projection-2} from Lemma \ref{prelim:lem-projection}, the global well-posedness (for  any fixed $N\geq 1$) follows exactly as in Subsection \ref{section:gwp}. Thus, it remains to prove the difference estimate \eqref{dynamics:eq-approximation}. Due to time-reflection symmetry, it suffices to treat the case $t\geq 0$. \\

From \eqref{dynamics:eq-psi-R} and \eqref{dynamics:eq-psi-R-N}, it follows that 
\begin{equation}
\psi - \psi^{(N)} = - \Duh_R \Big[ r^{-1} \Nl\big( r^{-1} \psi \big) - \Proj \Big( r^{-1} \Nl \big( r^{-1} \Proj \psi^{(N)} \big) \Big) \Big]. 
\end{equation}
Using Lemma \ref{dynamics:lem-linear}, we obtain that 
\begin{equation}\label{dynamics:eq-approximation-p0}
\begin{aligned}
&\, \Big\| (\psi, \partial_t \psi)(t) - (\psi^{(N)}, \partial_t \psi^{(N)})(t) \Big\|_{(\Czero_0 \times \Cmone)([1,R])} \\
\lesssim& \,  \Big\| \rho^{-1} \Nl\big( \rho^{-1} \psi \big) - 
\Proj \Big( \rho^{-1} \Nl \big( \rho^{-1} \Proj \psi^{(N)} \big) \Big) \Big\|_{L_s^1 L_\rho^2([0,t]\times [1,R])}. 
\end{aligned}
\end{equation}
We now decompose
\begin{align}
&\rho^{-1} \Nl\big( \rho^{-1} \psi \big) - 
\Proj \Big( \rho^{-1} \Nl \big( \rho^{-1} \Proj \psi^{(N)} \big) \Big)  \notag \\
=& \, \big( 1 -\Proj \big) \Big( \rho^{-1} \Nl\big( \rho^{-1} \psi \big) \Big) \label{dynamics:eq-approximation-p1} \\ 
+& \, \Proj \Big( \rho^{-1} \Nl\big( \rho^{-1} \psi \big) - \rho^{-1} \Nl\big( \rho^{-1} \Proj \psi \big) \Big) \label{dynamics:eq-approximation-p2} \\
+& \,  \Proj \Big( \rho^{-1} \Nl\big( \rho^{-1} \Proj \psi \big) - \rho^{-1} \Nl\big( \rho^{-1} \Proj \psi^{(N)} \big) \Big). 
\label{dynamics:eq-approximation-p3} 
\end{align}
Using Lemma \ref{prelim:lem-projection}, the Lipschitz continuity of $\Nl$, and the crude estimate $\rho^{-1}\leq 1$ for all $\rho \in [1,R]$, it easily follows that 
\begin{align*}
\big\|  \eqref{dynamics:eq-approximation-p1} \big\|_{L^2_\rho([1,R])} 
&\lesssim R^{1/2-\kappa} \big( R/N \big)^{\alpha} \big\| \psi \big\|_{\Czero([1,R])}, \\ 
\big\|  \eqref{dynamics:eq-approximation-p2} \big\|_{L^2_\rho([1,R])} 
&\lesssim R^{1/2-\kappa} \big( R/N \big)^{\alpha} \big\| \psi \big\|_{\Czero([1,R])}, \\ 
\big\|  \eqref{dynamics:eq-approximation-p3} \big\|_{L^2_\rho([1,R])} 
&\lesssim R^{1/2-\kappa}  \big\| \psi - \psi^{(N)} \big\|_{\Czero([1,R])}.
\end{align*}
Inserting this back into \eqref{dynamics:eq-approximation-p0}, it follows that
\begin{equation*}
\begin{aligned}
&\, \Big\| (\psi, \partial_t \psi)(t) - (\psi^{(N)}, \partial_t \psi^{(N)})(t) \Big\|_{(\Czero_0 \times \Cmone)([1,R])} \\
\lesssim& \,  R^{1/2-\kappa} \big( R/N \big)^{\alpha} \int_0^t \ds \,   
\big\| \psi(s) \big\|_{\Czero([1,R])}  + R^{1/2-\kappa} \int_0^t \ds \,  \big\| (\psi -\psi^{(N)})(s) \big\|_{\Czero([1,R])}. 
\end{aligned}
\end{equation*}
Using Gronwall's inequality, it follows that 
\begin{equation*}
\Big\| (\psi, \partial_t \psi)(t) - (\psi^{(N)}, \partial_t \psi^{(N)})(t) \Big\|_{(\Czero_0 \times \Cmone)([1,R])}  \leq C(R,T)  N^{-\alpha}
 \sup_{s\in [0,t]} \big\| \psi(s) \big\|_{\Czero([1,R])} . 
\end{equation*}
Together with Proposition \ref{dynamics:prop-gwp}, this implies the desired estimate. 
\end{proof}

Equipped with the approximation lemma (Lemma \ref{dynamics:lem-approximation}), we can now prove Proposition \ref{dynamics:prop-invariance-finite}. 

\begin{proof}[Proof of Proposition \ref{dynamics:prop-invariance-finite}]
Throughout the proof, we let $\alpha:=1/2-\delta$ and $\kappa:=-1/2-\delta$. 
We first recall that, as stated in Proposition \ref{dynamics:prop-gwp} and Lemma \ref{dynamics:lem-approximation},   \eqref{dynamics:eq-psi-R} and \eqref{dynamics:eq-psi-R-N} are globally well-posed on the support of the Gibbs measure and that the corresponding global flows are denoted by $\vec{\Psi}_R$ and $\vec{\Psi}_R^{(N)}$. In order to prove the proposition, we have to prove for all bounded, Lipschitz continuous $f \colon (\Czero_0 \times \Cmone)([1,R]) \rightarrow \R$ and all $t\in \R$ that 
\begin{equation}\label{dynamics:eq-invariance-R-p1}
\E_{\vec{\nu}_{n,k,R}} \big[ f \circ \vec{\Psi}_R(t) \big] 
= \E_{\vec{\nu}_{n,k,R}} \big[ f \big]. 
\end{equation}
To this end, we let $N\geq 1$ and decompose 
\begin{align}
\E_{\vec{\nu}_{n,k,R}} \big[ f \circ \vec{\Psi}_R(t) \big] 
- \E_{\vec{\nu}_{n,k,R}} \big[ f \big]
&= \E_{\vec{\nu}_{n,k,R}} \big[ f \circ \vec{\Psi}_R(t) \big] 
- \E_{\vec{\nu}_{n,k,R}} \big[ f \circ \vec{\Psi}_R^{(N)}(t) \big] 
\label{dynamics:eq-invariance-R-p2} \\ 
&+  \E_{\vec{\nu}_{n,k,R}} \big[ f \circ \vec{\Psi}_R^{(N)}(t) \big]  
-  \E_{\vec{\nu}_{n,k,R}^{(N)}} \big[ f \circ \vec{\Psi}_R^{(N)}(t) \big] 
\label{dynamics:eq-invariance-R-p3} \\ 
&+ \E_{\vec{\nu}_{n,k,R}^{(N)}} \big[ f \circ \vec{\Psi}_R^{(N)}(t) \big]  
- \E_{\vec{\nu}_{n,k,R}^{(N)}} \big[ f  \big] 
\label{dynamics:eq-invariance-R-p4} \\ 
&+\E_{\vec{\nu}_{n,k,R}^{(N)}} \big[ f  \big] 
- \E_{\vec{\nu}_{n,k,R}} \big[ f  \big] 
\label{dynamics:eq-invariance-R-p5}. 
\end{align}
The first term \eqref{dynamics:eq-invariance-R-p2} can be estimated using Lemma \ref{dynamics:lem-approximation}, the representation of $\vec{\nu}_{n,k,R}$ from Proposition \ref{Gibbs:prop-Gibbs}, and the moment bounds from Lemma \ref{bounds:lem-exponential}. The second term \eqref{dynamics:eq-invariance-R-p3} can be estimated using Proposition \ref{Gibbs:prop-Gibbs}. The third term \eqref{dynamics:eq-invariance-R-p4} vanishes due to the invariance of $\vec{\nu}_{n,k,R}^{(N)}$ under $\vec{\Psi}_R^{(N)}$, which follows from ODE-results. Finally, the fourth term \eqref{dynamics:eq-invariance-R-p5} can be estimated using Proposition \ref{Gibbs:prop-Gibbs}. In total, it follows that 
\begin{equation*}
\Big| \E_{\vec{\nu}_{n,k,R}} \big[ f \circ \vec{\Psi}_R(t) \big] 
- \E_{\vec{\nu}_{n,k,R}} \big[ f \big] \Big| 
\lesssim_{f,R,t} \liminf_{N\rightarrow \infty} \Big( N^{-1/2+\delta} 
+ \big\| \vec{\nu}_{n,k,R}^{\,(N)} - \vec{\nu}_{n,k,R} \big\|_{\textup{TV}} \Big) =0, 
\end{equation*}
which yields \eqref{dynamics:eq-invariance-R-p1}. 
\end{proof}

\subsection{Invariance for the semi-infinite interval}
\label{section:invariance-semi-infinite}

In this subsection, we prove the invariance of the Gibbs measures for the semi-infinite interval.

\begin{proposition}[Invariance for the semi-infinite interval]\label{dynamics:prop-invariance}
Let $n\geq 0$ and let $k\geq 1$. Then, the Gibbs measure $\vec{\nu}_{n,k}$ is invariant under the dynamics of \eqref{dynamics:eq-psi-infty}. 
\end{proposition}

The main ingredients in the following proof are the weak convergence of the Gibbs measures $\vec{\nu}_{n,k,R}$ as $R\rightarrow \infty$ (Proposition \ref{Gibbs:prop-Gibbs}), the invariance of the Gibbs measures for finite intervals (Proposition \ref{dynamics:prop-invariance-finite}), and finite speed of propagation. 

\begin{proof}[Proof of Proposition \ref{dynamics:prop-invariance}]
Let $\vec{\Psi}$ and $\vec{\Psi_R}$ be the global flows from \eqref{dynamics:eq-flow-psi} and \eqref{dynamics:eq-flow-psi-R}. In order to prove the invariance of $\vec{\nu}_{n,k}$, it suffices to prove for all $t\in \R$ and $K\geq 1$ that 
\begin{equation}\label{dynamics:eq-invariance-p1}
\big( \vecRest[K][\infty] \big)_\# \vec{\Psi}(t)_\# \vec{\nu}_{n,k} = \big( \vecRest[K][\infty] \big)_\#  \vec{\nu}_{n,k},
\end{equation}
which are viewed as measures on $(\Czero_{(0)} \times \Cmone)([1,K])$ with $\alpha:=1/2-\delta$ and $\kappa:=-1/2-\delta$. In order to utilize our earlier results, we need to insert additional restriction operators. To this end, we let $L,M\geq 1$ satisfy
\begin{equation}\label{dynamics:eq-invariance-p2}
K + |t| + 1 \leq L \leq M. 
\end{equation}
Due to finite speed of propagation, it holds that 
\begin{equation}\label{dynamics:eq-invariance-p3}
\vecRest[K][\infty] \circ \vec{\Psi}(t) = \vecRest[K][L] \circ \vec{\Psi}_L(t) \circ \vecRestZ[L][M] \circ \vecRest[M][\infty]. 
\end{equation}
The operator $\vecRestZ[L][M]$ is as in Definition \ref{prelim:def-restriction} and guarantees that the argument of $\vec{\Psi}_L(t)$ satisfies the zero Dirichlet boundary conditions. Using  \eqref{dynamics:eq-invariance-p3} and the weak convergence of $\vec{\nu}_{n,k,R}$ (as in Corollary \ref{Gaussian:cor-white-noise} and Proposition \ref{Gibbs:prop-Gibbs}), it follows that 
\begin{equation}\label{dynamics:eq-invariance-p4}
\begin{aligned}
\big( \vecRest[K][\infty] \big)_\# \vec{\Psi}(t)_\# \vec{\nu}_{n,k} 
&=  \big( \vecRest[K][L] \circ \vec{\Psi}_L(t) \circ \vecRestZ[L][M] \big)_\# \big( \vecRest[M][\infty] \big)_\# \vec{\nu}_{n,k} \\
&= \operatorname{w-lim}\displaylimits_{R\rightarrow \infty}  \big( \vecRest[K][L] \circ \vec{\Psi}_L(t) \circ \vecRestZ[L][M] \big)_\#  \big( \vecRest[M][R] \big)_\# \vec{\nu}_{n,k,R}. 
\end{aligned}
\end{equation}
The second identity in \eqref{dynamics:eq-invariance-p4} involves the weak limit on $(\Czero_{(0)} \times \Cmone)([1,K])$. Provided that $R\geq M$, we also have the identity
\begin{equation}\label{dynamics:eq-invariance-p5}
 \vecRest[K][L] \circ \vec{\Psi}_L(t) \circ \vecRestZ[L][M] \circ \vecRest[M][R] = \vecRest[K][R] \circ \vec{\Psi}_R(t),
\end{equation}
which is similar to \eqref{dynamics:eq-invariance-p3}. From \eqref{dynamics:eq-invariance-p5}, it follows that 
\begin{equation}\label{dynamics:eq-invariance-p6} 
\begin{aligned}
\operatorname{w-lim}\displaylimits_{R\rightarrow \infty}  \big( \vecRest[K][L] \circ \vec{\Psi}_L(t) \circ \vecRestZ[L][M] \big)_\#  \big( \vecRest[M][R] \big)_\# \vec{\nu}_{n,k,R}
= \operatorname{w-lim}\displaylimits_{R\rightarrow \infty} \big(\vecRest[K][R]\big)_{\#} \vec{\Psi}_R(t)_\# \vec{\nu}_{n,k,R}. 
\end{aligned}
\end{equation}
Using the invariance of the Gibbs measure for finite intervals (Proposition \ref{dynamics:prop-invariance-finite}), we obtain that 
\begin{equation}\label{dynamics:eq-invariance-p7} 
\begin{aligned}
 \operatorname{w-lim}\displaylimits_{R\rightarrow \infty} \big(\vecRest[K][R]\big)_{\#} \vec{\Psi}_R(t)_\# \vec{\nu}_{n,k,R}
 = \operatorname{w-lim}\displaylimits_{R\rightarrow \infty} \big(\vecRest[K][R]\big)_{\#}  \vec{\nu}_{n,k,R}.
\end{aligned}
\end{equation}
By using the weak convergence of $\vec{\nu}_{n,k,R}$ (as in Corollary \ref{Gaussian:cor-white-noise} and Proposition \ref{Gibbs:prop-Gibbs}) for a second time, it follows that
\begin{equation}\label{dynamics:eq-invariance-p8} 
\begin{aligned}
\operatorname{w-lim}\displaylimits_{R\rightarrow \infty} \big(\vecRest[K][R]\big)_{\#}  \vec{\nu}_{n,k,R}
= \big(\vecRest[K][\infty]\big)_{\#}  \vec{\nu}_{n,k}. 
\end{aligned}
\end{equation}
The desired identity \eqref{dynamics:eq-invariance-p1} now follows by combining  \eqref{dynamics:eq-invariance-p4}, \eqref{dynamics:eq-invariance-p6}, \eqref{dynamics:eq-invariance-p7}, and \eqref{dynamics:eq-invariance-p8}, which completes our argument.  
\end{proof}

\section{Proof of Theorem \ref{intro:thm-main} and Corollary \ref{intro:cor-failure-soliton}}\label{section:proof}

In this section, we prove the main results of this article. Due to our earlier lemmas and propositions from Section \ref{section:Gaussian}, Section \ref{section:Gibbs}, and Section \ref{section:dynamics}, the remaining proofs are relatively short.

\begin{proof}[Proof of Theorem \ref{intro:thm-main}]
We rigorously define the Gibbs measure $\vec{\mu}_{n,k}$ as the push-forward of $\vec{\nu}_{n,k}$, which has been constructed in Proposition \ref{Gibbs:prop-Gibbs}, under the transformation 
\begin{equation*}
\big( \psi_0, \psi_1 \big) \mapsto \big( \phi_0 , \phi_1 \big) := \big( Q_{n,k}+ r \psi_0, r \psi_1 \big). 
\end{equation*}
From the definition (and Proposition \ref{Gibbs:prop-Gibbs}), it directly follows that $\vec{\mu}_{n,k}$ is supported on the state space $\State$. 
Using the change of variables from \eqref{prelim:eq-change-of-variables}, the global well-posedness of \eqref{intro:eq-phi} and the invariance of the Gibbs measures follows directly from Proposition \ref{dynamics:prop-gwp} and Proposition \ref{dynamics:prop-invariance}, respectively. 
\end{proof}

It remains to prove Corollary \ref{intro:cor-failure-soliton}, which essentially follows from Theorem \ref{intro:thm-main} (or Proposition \ref{dynamics:prop-invariance}) and the Poincar\'{e} recurrence theorem.

\begin{proof}[Proof of Corollary \ref{intro:cor-failure-soliton}] 
Throughout this proof, we work in the unknown $\psi$ from \eqref{prelim:eq-change-of-variables}. In this unknown, the linearized equation \eqref{intro:eq-phi-lin} takes the form
\begin{equation}\label{proof:eq-p1}
\partial_t^2 \psi_{\textup{lin}} - \partial_r^2 \psi_{\textup{lin}} + \frac{k(k+1)}{r^2} \cos\big( 2Q_{n,k}\big) \psi_{\textup{lin}} = 0. 
\end{equation}
To simplify the notation, we let $\alpha:=1/2-\delta$, let $\kappa:=-1/2-\delta$, let $\vec{\Psi}$ be the global flow from \eqref{dynamics:eq-flow-psi}, and let $\vec{\Psi}_{\textup{lin}}$ be the global flow of \eqref{proof:eq-p1}. By time-reversal symmetry, it suffices to prove the claim in \eqref{intro:eq-failure-soliton} for $t\rightarrow \infty$.
Thus, it remains to prove that  
\begin{equation}\label{proof:eq-p2}
\inf_{(\psi^+_0,\psi^+_1)} \limsup_{t\rightarrow \infty} 
\big\| \vec{\Psi}(t) (\psi_0,\psi_1) - \vec{\Psi}_{\textup{lin}}(t) (\psi_0^+,\psi_1^+) \big\|_{(\Czero \times \Cmone)([1,2])} >0
\end{equation}
holds $\vec{\nu}_{n,k}$-almost surely, 
where the infimum is taken over all $(\psi_0^+,\psi_1^+)\in (\Czero_0 \times \Cmone)([1,\infty))$. We first show that \eqref{proof:eq-p2} follows from a simpler statement which does not explicitly involve $(\psi_0^+,\psi_1^+)$. For any $t\in \R$ and $\tau \in[0,1/4]$, it follows from the group properties of the flows $\vec{\Psi}$ and $\vec{\Psi}_{\textup{lin}}$, finite speed of propagation, and the boundedness of $\vec{\Psi}_{\textup{lin}}$ (as in Lemma \ref{dynamics:lem-linear}) that 
\begin{align*}
&
\big\| 
\vec{\Psi}_{\textup{lin}}(-\tau) \vec{\Psi}(\tau) \vec{\Psi}(t) (\psi_0,\psi_1) 
-  \vec{\Psi}(t) (\psi_0,\psi_1)  \big\|_{(\Czero \times \Cmone)([1,3/2])}\\
\leq&\, 
\big\| 
\vec{\Psi}_{\textup{lin}}(-\tau)  \vec{\Psi}(t+\tau) (\psi_0,\psi_1) 
-  \vec{\Psi}_{\textup{lin}}(-\tau) \vec{\Psi}_{\textup{lin}}(t+\tau) (\psi_0^+,\psi_1^+)  \big\|_{(\Czero \times \Cmone)([1,3/2])}\\ 
+&\, 
\big\| 
 \vec{\Psi}_{\textup{lin}}(-\tau) \vec{\Psi}_{\textup{lin}}(t+\tau) (\psi_0^+,\psi_1^+) 
 - \vec{\Psi}_{\textup{lin}}(t) (\psi_0^+,\psi_1^+)\big\|_{(\Czero \times \Cmone)([1,3/2])} \\
 +& \, 
\big\| 
 \vec{\Psi}_{\textup{lin}}(t) (\psi_0^+,\psi_1^+) - \vec{\Psi}(t) (\psi_0,\psi_1) \big\|_{(\Czero \times \Cmone)([1,3/2])}  \\
 \lesssim& \, 
\big\| 
 \vec{\Psi}(t+\tau) (\psi_0,\psi_1) 
-   \vec{\Psi}_{\textup{lin}}(t+\tau) (\psi_0^+,\psi_1^+)  \big\|_{(\Czero \times \Cmone)([1,2])} \\
 +& \, 
\big\| 
   \vec{\Psi}(t) (\psi_0,\psi_1) 
   - \vec{\Psi}_{\textup{lin}}(t) (\psi_0^+,\psi_1^+) 
   \big\|_{(\Czero \times \Cmone)([1,2])}.
\end{align*}
As a result, it suffices to prove that
\begin{equation}\label{proof:eq-p3}
\limsup_{t\rightarrow \infty} \sup_{\tau \in [0,1/4]} 
\big\| 
\vec{\Psi}_{\textup{lin}}(-\tau) \vec{\Psi}(\tau) \vec{\Psi}(t) (\psi_0,\psi_1) 
-  \vec{\Psi}(t) (\psi_0,\psi_1)  \big\|_{(\Czero \times \Cmone)([1,3/2])} >0
\end{equation}
holds $\vec{\nu}_{n,k}$-almost surely. 
To this end, we let $\epsilon>0$ be arbitrary and define the event 
\begin{align*}
A_{n,k,\epsilon} :=& \Big\{ (\psi_0,\psi_1)\in (\Czero_0 \times \Cmone)([1,\infty))\colon  \\
& \sup_{\tau \in [0,1/4]} 
\big\| 
\vec{\Psi}_{\textup{lin}}(-\tau) \vec{\Psi}(\tau)  (\psi_0,\psi_1) 
-  (\psi_0,\psi_1)  \big\|_{(\Czero \times \Cmone)([1,3/2])} 
\geq \epsilon \Big\}.  
\end{align*}
Using the invariance of $\vec{\nu}_{n,k}$ (as in Proposition \ref{dynamics:prop-invariance}) and Poincar\'{e}'s recurrence theorem, it follows 
that there exists a set $B_{n,k,\epsilon}\subseteq A_{n,k,\epsilon}$ such that $\vec{\nu}_{n,k}(B_{n,k,\epsilon})=\vec{\nu}_{n,k}(A_{n,k,\epsilon})$ and such that, for all $(\psi_0,\psi_1)\in B_{n,k,\epsilon}$, it holds that $\vec{\Psi}(j)(\psi_0,\psi_1)\in A_{n,k,\epsilon}$ for infinitely many integers $j\geq 1$. In particular, it holds  for  all $(\psi_0,\psi_1) \in B_{n,k,\epsilon}$ that
\begin{equation}\label{proof:eq-p5}
\limsup_{j\rightarrow \infty} \sup_{\tau \in [0,1/4]} 
\big\| 
\vec{\Psi}_{\textup{lin}}(-\tau) \vec{\Psi}(\tau) \vec{\Psi}(j) (\psi_0,\psi_1) 
-  \vec{\Psi}(j) (\psi_0,\psi_1)  \big\|_{(\Czero \times \Cmone)([1,3/2])} \geq \epsilon
\end{equation}
and thus \eqref{proof:eq-p3} is satisfied. It therefore only remains to prove that
\begin{equation}\label{proof:eq-p6}
\lim_{\epsilon \downarrow 0} \vec{\nu}_{n,k}\big( A_{n,k,\epsilon} \big) =1. 
\end{equation}
Since $\nu_{n,k}$ is absolutely continuous with respect to the Gaussian measure $\scrg_{n,k}$ (Proposition \ref{Gibbs:prop-Gibbs}), it is clear that $\vec{\nu}_{n,k}$-almost surely the nonlinearity
\begin{equation*}
\sin\big( 2 (Q_{n,k} + r^{-1} \psi_0)\big) 
- \sin\big( 2 Q_{n,k}\big) 
- 2r^{-1} \cos\big( 2 Q_{n,k} \big) \psi_0
\end{equation*}
is not identically zero on the spatial interval $[1,5/4]$. Together with local well-posedness, this implies that
\begin{equation}\label{proof:eq-p7}
\big\| 
\vec{\Psi}_{\textup{lin}}(-\tau) \vec{\Psi}(\tau)  (\psi_0,\psi_1) 
-  (\psi_0,\psi_1)  \big\|_{(\Czero \times \Cmone)([1,3/2])} >0 
\end{equation}
holds $\vec{\nu}_{n,k}$-almost surely. Using the continuity from below of the probability measure $\vec{\nu}_{n,k}$, this implies \eqref{proof:eq-p6} and therefore completes the proof.
\end{proof}
\begin{remark}\label{proof:rem-generalization}
The proof of Corollary \ref{intro:cor-failure-soliton} primarily uses the invariance of $\vec{\nu}_{n,k}$ under $\vec{\Psi}$, the group properties of $\vec{\Psi}$ and $\vec{\Psi}_{\textup{lin}}$, and the boundedness of $\vec{\Psi}_{\textup{lin}}$ on $\Czero_0 \times \Cmone$. 
In order to obtain $\eqref{intro:eq-failure-soliton}$ on the interval $[1,2]$ rather than the whole interval $[1,\infty)$, we also used finite speed of propagation. All of these ingredients (except possibly invariance) are available in many situations, and our proof can easily be generalized to different flows than $\vec{\Psi}_{\textup{lin}}$ and other norms than $\Czero_0 \times \Cmone$. 
\end{remark}

\begin{appendix}

\section{Elements of probability theory}\label{section:probability}

In this appendix, we recall results from probability theory. To this end, we let $(\Omega,\mathcal{F},\mathbb{P})$ be a probability space and let $\mathbb{E}$ be the corresponding expectation operator. 

\begin{lemma}[Gaussian hypercontractivity]\label{prelim:lem-properties-Gaussian}
Let $g$ be a Gaussian random variables and let $p\geq 1$. Then, it holds that
\begin{equation*}
\E \big[ |g|^p \big]^{1/p} \lesssim \sqrt{p} \, \E \big[ g^2 \big]^{1/2}. 
\end{equation*}
\end{lemma}
We remark that Gaussian hypercontractivity is a much more general phenomenon than stated in Lemma \ref{prelim:lem-properties-Gaussian}, since it also applies to polynomials in infinitely many Gaussian variables. 
In the next lemma, we recall a version of Kolmogorov's continuity theorem, which is used to obtain the growth and Hölder estimates in Section \ref{section:Gaussian}. 

\begin{lemma}[{Kolmogorov's continuity theorem \cite[p.182]{S11}}]\label{prelim:lem-kolmogorov}
Let $T\geq 1$ and let $(X(t))_{1\leq t \leq T}$ be a continuous stochastic process. Assume that there exist $C>0$, $p\geq 1$, and $\alpha \in (0,1]$ such that the estimate
\begin{equation*}
\E \Big[ \big| X(t) - X(s) \big|^p \Big]^{1/p} \leq C |t-s|^{\frac{1}{p}+\alpha} 
\end{equation*}
is satisfied for all $1\leq s,t \leq T$. Then, it holds for all $0<\beta<\alpha$ that 
\begin{equation*}
\E \Big[ \sup_{\substack{1\leq s,t \leq T\colon \\ s\neq t }} 
\bigg( \frac{|X(t)-X(s)|}{|t-s|^\beta} \bigg)^p \bigg]^{1/p} 
\leq \frac{5}{(1-2^{-\alpha})(1-2^{\beta-\alpha})}  C T^{\frac{1}{p}+\alpha-\beta}.  
\end{equation*}
\end{lemma}

We now recall an estimate for the Laplace-transform of Gaussian measures, which is derived from the Bou\'{e}-Dupuis formula. For the sake of simplicity, we directly restrict ourselves to the setting of Section \ref{section:Gaussian}
and Section \ref{section:Gibbs}. 

\begin{lemma}[Consequence of Bou\'{e}-Dupuis formula]\label{prelim:lem-boue-dupuis}
Let $n\geq 0$ and $k\geq 1$. Furthermore, let $R\geq R_0$, let $0<\delta\ll 1$, let $\alpha=1/2-\delta$, and let $\kappa=-1/2-\delta$. Finally, let
\begin{equation*}
V\colon \Czero([1,R]) \rightarrow \R 
\end{equation*}
be continuous and integrable with respect to $\scrg_{n,k,R}$.  Then, it holds that
\begin{equation}\label{prelim:eq-boue-dupuis}
\begin{aligned}
&- \log \Big( \E_{\scrg_{n,k,R}} \Big[ \exp\big( - V(\psi) \big) \Big] \Big) \\ 
\geq& \,  \E_{\scrg_{n,k,R}} \Big[ \inf_{\zeta \in \dot{H}^1_0([1,R])} \Big\{ V\big( \psi + \zeta \big) 
+ \frac{1}{2}  \big\langle \zeta , A_{n,k,R} \zeta \big\rangle_{L^2([1,R])} \Big\}  \Big]. 
\end{aligned}
\end{equation}
\end{lemma}

\begin{proof} In order to use the Bou\'{e}-Dupuis formula \cite{BD98,BG18,HW22}, we first introduce additional notation. 
We let $(\Omega,\mathcal{F},\mathscr{G}_{n,k,R})$ be a sufficiently rich probability space and let $\Psi \colon [0,1] \times [1,R]\rightarrow \R$ be a Gaussian process satisfying 
\begin{equation*}
\E_{\mathscr{G}_{n,k,R}}\Big[ \Psi(t,r) \Psi(s,\rho)\Big] = \min(t,s) \,  G_{n,k,R}(r,\rho)
\end{equation*}
for all $t,s\in [0,1]$ and $r,\rho \in [1,R]$. In particular, it holds that 
\begin{equation}\label{prelim:eq-boue-dupuis-p1}
\textup{Law}_{\mathscr{G}_{n,k,R}}\big( \Psi(1) \big) = \scrg_{n,k,R}. 
\end{equation}
We let $(\mathcal{F}_t)_{t\in [0,1]}$ be the augmented, natural filtration associated with the Gaussian process $\Psi$. Furthermore, we let $\dot{\mathbb{H}}_0^1([0,1]\times [1,R])$ be the space of progressively measurable functions $z \colon [0,1] \rightarrow \dot{H}_0^1([1,R])$ satisfying 
\begin{equation*}
\int_0^1 \dt \, \big\langle z(t) , A_{n,k,R} z(t) \rangle_{L^2([1,R])} < \infty \qquad \mathscr{G}_{n,k,R}\textup{-almost surely.}
\end{equation*}
For any $z\in \dot{\mathbb{H}}_0^1([0,1]\times [1,R])$, we define  
\begin{equation*}
Z(t) := \int_0^t \ds \, z(s). 
\end{equation*}
We now let $M\geq 1$ be arbitrary and define $V_M := \max(V,-M)$, which is bounded below. Using  the Bou\'{e}-Dupuis formula (as stated in \cite[Theorem 1.1]{HW22}), it follows that 
\begin{equation}\label{prelim:eq-boue-dupuis-p2} 
\begin{aligned}
&- \log \Big( \E_{\scrg_{n,k,R}} \Big[ \exp\big( - V_M(\psi) \big) \Big] \Big) \\ 
=& \,  - \log \Big( \E_{\mathscr{G}_{n,k,R}} \Big[ \exp\big( - V_M(\Psi(1)) \big) \Big] \Big) \\
=& \,  \inf_{z\in \dot{\mathbb{H}}_0^1} \E_{\mathscr{G}_{n,k,R}} \Big[ V_M ( \Psi(1) + Z(1) ) 
+ \frac{1}{2} \int_0^1 \dt \,  \big\langle z(t) , A_{n,k,R} z(t) \rangle_{L^2([1,R])} \Big]. 
\end{aligned}
\end{equation}
Using the triangle inequality and Cauchy-Schwarz, it holds that 
\begin{align*}
 \big\langle Z(1) , A_{n,k,R} Z(1) \rangle_{L^2([1,R])}^{1/2}
 &\leq \int_0^1 \dt \,  \big\langle z(t) , A_{n,k,R} z(t) \rangle_{L^2([1,R])}^{1/2} \\ 
 &\leq \Big( \int_0^1 \dt \,  \big\langle z(t) , A_{n,k,R} z(t) \rangle_{L^2([1,R])} \Big)^{1/2}. 
\end{align*}
Combined with the trivial estimate $V_M \geq V$ and $Z(1) \in \dot{H}_0^1([1,R])$, it follows that 
\begin{equation}\label{prelim:eq-boue-dupuis-p3}
\begin{aligned}
 & \inf_{z\in \dot{\mathbb{H}}_0^1} \E_{\mathscr{G}_{n,k,R}} \Big[ V_M ( \Psi(1) + Z(1) ) 
+ \frac{1}{2} \int_0^1 \dt \,  \big\langle z(t) , A_{n,k,R} z(t) \rangle_{L^2([1,R])} \Big] \\ 
\geq&\, \inf_{z\in \dot{\mathbb{H}}_0^1} \E_{\mathscr{G}_{n,k,R}} \Big[ V ( \Psi(1) + Z(1) ) 
+ \frac{1}{2}  \big\langle Z(1) , A_{n,k,R} Z(1) \rangle_{L^2([1,R])} \Big] \\ 
\geq& \, \E_{\mathscr{G}_{n,k,R}} \Big[ \inf_{\zeta \in \dot{H}^1_0([1,R])} \Big\{ V\big( \Psi(1) + \zeta \big) 
+ \frac{1}{2}  \big\langle \zeta , A_{n,k,R} \zeta \big\rangle_{L^2([1,R])} \Big\}  \Big].
\end{aligned}
\end{equation}
By combining \eqref{prelim:eq-boue-dupuis-p1}, \eqref{prelim:eq-boue-dupuis-p2}, and \eqref{prelim:eq-boue-dupuis-p3}, it follows that 
\begin{equation}\label{prelim:eq-boue-dupuis-p4}
\begin{aligned}
&- \log \Big( \E_{\scrg_{n,k,R}} \Big[ \exp\big( - V_M(\psi) \big) \Big] \Big) \\ 
\geq& \,  \E_{\scrg_{n,k,R}} \Big[ \inf_{\zeta \in \dot{H}^1_0([1,R])} \Big\{ V\big( \psi + \zeta \big) 
+ \frac{1}{2}  \big\langle \zeta , A_{n,k,R} \zeta \big\rangle_{L^2([1,R])} \Big\}  \Big]. 
\end{aligned}
\end{equation}
By letting $M\rightarrow \infty$ and using monotone convergence, this implies \eqref{prelim:eq-boue-dupuis}.
\end{proof}

At the end of this appendix, we recall the definition of weak convergence for probability measures on metric spaces. 

\begin{definition}[Weak convergence]\label{appendix:def-weak-convergence} 
Let $X$ be a metric space and let $\Sigma$ be the corresponding Borel $\sigma$-algebra. Furthermore, let $(\lambda_R)_{R\geq 1}$ be a family of probability measures on $(X,\Sigma)$ and let $\lambda$ be a probability measure on $(X,\Sigma)$. Then, we say that $(\lambda_R)_{R\geq 1}$ converges weakly to $\lambda$ on $X$ if and only if 
\begin{equation}
\lim_{R\rightarrow \infty} \int_X f(\psi) \mathrm{d}\lambda_R(\psi) = \int_X f(\psi) \mathrm{d}\lambda(\psi)
\end{equation}
for all bounded and Lipschitz continuous $f\colon X \rightarrow \R$. 
\end{definition}

\begin{remark}
In most articles and textbooks, the metric space $X$ is fixed. In this article, however, $X$ is not fixed (see e.g. Proposition \ref{Gibbs:prop-Gibbs}). This is the reason for adding the phrase ``\textit{on X}" in Definition \ref{appendix:def-weak-convergence}. 
\end{remark}

\end{appendix}

\bibliography{Equivariant_Library}

\begin{thebibliography}{BDNY22}

\bibitem[AK20]{AK20}
S. Albeverio and S. Kusuoka.
\newblock The invariant measure and the flow associated to the
  {$\Phi^4_3$}-quantum field model.
\newblock {\em Ann. Sc. Norm. Super. Pisa Cl. Sci. (5)}, 20(4):1359--1427,
  2020.

\bibitem[BSSS92]{BSSS92}
B.~S. Balakrishna, V. Sanyuk, J. Schechter, and A. Subbaraman.
\newblock Cutoff quantization and the skyrmion.
\newblock {\em Phys. Rev. D}, 45:344--351, Jan 1992.

\bibitem[BG20]{BG18}
N. Barashkov and M. Gubinelli.
\newblock A variational method for {$\Phi^4_3$}.
\newblock {\em Duke Math. J.}, 169(17):3339--3415, 2020.

\bibitem[BCM12]{BCM12}
P. Bizo\'{n}, T. Chmaj, and M. Maliborski.
\newblock Equivariant wave maps exterior to a ball.
\newblock {\em Nonlinearity}, 25(5):1299--1309, 2012.

\bibitem[BD98]{BD98}
M. Bou\'{e} and P. Dupuis.
\newblock A variational representation for certain functionals of {B}rownian
  motion.
\newblock {\em Ann. Probab.}, 26(4):1641--1659, 1998.

\bibitem[Bou94]{B94}
J. Bourgain.
\newblock Periodic nonlinear {S}chr\"{o}dinger equation and invariant measures.
\newblock {\em Comm. Math. Phys.}, 166(1):1--26, 1994.

\bibitem[Bou00]{B00}
J. Bourgain.
\newblock Invariant measures for {NLS} in infinite volume.
\newblock {\em Comm. Math. Phys.}, 210(3):605--620, 2000.

\bibitem[Bou96]{B96}
J. Bourgain.
\newblock Invariant measures for the {$2$}{D}-defocusing nonlinear
  {S}chr\"{o}dinger equation.
\newblock {\em Comm. Math. Phys.}, 176(2):421--445, 1996.

\bibitem[{Bri}20]{B20II}
B. {Bringmann}.
\newblock {Invariant Gibbs measures for the three-dimensional wave equation
  with a Hartree nonlinearity II: Dynamics}.
\newblock arXiv:2009.04616, September 2020.
\newblock To appear in JEMS.

\bibitem[Bri22]{B20I}
B. Bringmann.
\newblock Invariant {G}ibbs measures for the three-dimensional wave equation
  with a {H}artree nonlinearity {I}: Measures.
\newblock {\em Stoch. Partial Differ. Equ. Anal. Comput.}, 10(1):1--89, 2022.

\bibitem[BDNY22]{BDNY22}
B. {Bringmann}, Y. {Deng}, A.~R. {Nahmod}, and H. {Yue}.
\newblock {Invariant Gibbs measures for the three dimensional cubic nonlinear
  wave equation}.
\newblock arXiv:2205.03893, May 2022.

\bibitem[CH53]{CH53}
R. Courant and D. Hilbert.
\newblock {\em Methods of mathematical physics. {V}ol. {I}}.
\newblock Interscience Publishers, Inc., New York, N.Y., 1953.

\bibitem[DNY19]{DNY19}
Y. {Deng}, A.~R. {Nahmod}, and H. {Yue}.
\newblock {Invariant Gibbs measures and global strong solutions for nonlinear
  Schr{\"o}dinger equations in dimension two}.
\newblock arXiv:1910.08492, October 2019.

\bibitem[DNY21]{DNY21}
Y. Deng, A.~R. Nahmod, and H. Yue.
\newblock Invariant {G}ibbs measure and global strong solutions for the
  {H}artree {NLS} equation in dimension three.
\newblock {\em J. Math. Phys.}, 62(3):Paper No. 031514, 39, 2021.

\bibitem[DNY20]{DNY20}
Y. {Deng}, A.~R. {Nahmod}, and H. {Yue}.
\newblock {Random tensors, propagation of randomness, and nonlinear dispersive
  equations}.
\newblock {\em Invent. math.}, 228:539--686, 2022.

\bibitem[FO76]{FO76}
J.~S. Feldman and K. Osterwalder.
\newblock The {W}ightman axioms and the mass gap for weakly coupled {$(\Phi
  ^{4})_{3}$} quantum field theories.
\newblock {\em Ann. Physics}, 97(1):80--135, 1976.

\bibitem[GJ87]{GJ87}
J. Glimm and A. Jaffe.
\newblock {\em Quantum physics}.
\newblock Springer-Verlag, New York, second edition, 1987.
\newblock A functional integral point of view.

\bibitem[GH21]{GH21}
M. Gubinelli and M. Hofmanov\'{a}.
\newblock A {PDE} construction of the {E}uclidean {$\phi_3^4$} quantum field
  theory.
\newblock {\em Comm. Math. Phys.}, 384(1):1--75, 2021.

\bibitem[GKO18]{GKO18}
M. {Gubinelli}, H. {Koch}, and T. {Oh}.
\newblock {Paracontrolled approach to the three-dimensional stochastic
  nonlinear wave equation with quadratic nonlinearity}.
\newblock arXiv:1811.07808, November 2018.
\newblock To appear in JEMS.

\bibitem[HW22]{HW22}
Y. Hariya and S. Watanabe.
\newblock The {B}ou\'{e}-{D}upuis formula and the exponential
  hypercontractivity in the {G}aussian space.
\newblock {\em Electron. Commun. Probab.}, 27:Paper No. 18, 13, 2022.

\bibitem[JL21]{JL21}
J. {Jendrej} and A. {Lawrie}.
\newblock {Soliton resolution for energy-critical wave maps in the equivariant
  case}.
\newblock arXiv:2106.10738, June 2021.

\bibitem[KLLS15]{KLLS15}
C. Kenig, A. Lawrie, B. Liu, and W. Schlag.
\newblock Stable soliton resolution for exterior wave maps in all equivariance
  classes.
\newblock {\em Adv. Math.}, 285:235--300, 2015.

\bibitem[KLS14]{KLS14}
C.~E. Kenig, A. Lawrie, and W. Schlag.
\newblock Relaxation of wave maps exterior to a ball to harmonic maps for all
  data.
\newblock {\em Geom. Funct. Anal.}, 24(2):610--647, 2014.

\bibitem[LS13]{LS13}
A. Lawrie and W. Schlag.
\newblock Scattering for wave maps exterior to a ball.
\newblock {\em Adv. Math.}, 232:57--97, 2013.

\bibitem[MW20]{MW20}
A. Moinat and H. Weber.
\newblock Space-time localisation for the dynamic {$\Phi^4_3$} model.
\newblock {\em Comm. Pure Appl. Math.}, 73(12):2519--2555, 2020.

\bibitem[MW17]{MW17}
J.-C. Mourrat and H. Weber.
\newblock The dynamic {$\Phi^4_3$} model comes down from infinity.
\newblock {\em Comm. Math. Phys.}, 356(3):673--753, 2017.

\bibitem[OOT21]{OOT21}
T. {Oh}, M. {Okamoto}, and L. {Tolomeo}.
\newblock {Stochastic quantization of the $\Phi^3_3$-model}.
\newblock arXiv:2108.06777, August 2021.

\bibitem[PW81]{PW81}
G. Parisi and Y.~S. Wu.
\newblock Perturbation theory without gauge fixing.
\newblock {\em Sci. Sinica}, 24(4):483--496, 1981.

\bibitem[Sha88]{S88}
J. Shatah.
\newblock Weak solutions and development of singularities of the {${\rm
  SU}(2)$} {$\sigma$}-model.
\newblock {\em Comm. Pure Appl. Math.}, 41(4):459--469, 1988.

\bibitem[SS98]{SS98}
J. Shatah and M. Struwe.
\newblock {\em Geometric wave equations}, volume~2 of {\em Courant Lecture
  Notes in Mathematics}.
\newblock New York University, Courant Institute of Mathematical Sciences, New
  York; American Mathematical Society, Providence, RI, 1998.

\bibitem[Sky61]{S61}
T.~H.~R. Skyrme.
\newblock A non-linear field theory.
\newblock {\em Proceedings of the Royal Society of London. Series A.
  Mathematical and Physical Sciences}, 260(1300):127--138, 1961.

\bibitem[Str11]{S11}
D.~W. Stroock.
\newblock {\em Probability theory}.
\newblock Cambridge University Press, Cambridge, second edition, 2011.
\newblock An analytic view.

\bibitem[TW23]{TW+}
L. {Tolomeo} and H. {Weber}.
\newblock {Phase transition for invariant measures of the focusing
  Schr{\"o}dinger equation}.
\newblock arXiv:2306.07697, June 2023.

\bibitem[{Xu}14]{X14}
S. {Xu}.
\newblock {Invariant Gibbs Measure for 3D NLW in Infinite Volume}.
\newblock arXiv:1405.3856, May 2014.

\bibitem[Zhi94]{Z94}
P.~E. Zhidkov.
\newblock An invariant measure for a nonlinear wave equation.
\newblock {\em Nonlinear Anal.}, 22(3):319--325, 1994.

\end{thebibliography}
\bibliographystyle{myalpha}

\end{document}